\documentclass[final,leqno]{siamltex704}
\usepackage{amsmath}
\usepackage{graphicx}
\usepackage[notcite,notref]{showkeys}
\usepackage{mathrsfs}
\usepackage{color}
\usepackage{float}
\usepackage{amsfonts,amssymb}
\usepackage{dsfont}
\usepackage{pifont}
\usepackage{wrapfig} 
\usepackage{hyperref}
\usepackage{multirow} 
\numberwithin{equation}{section}
\usepackage{threeparttable}

\def\3bar{{|\hspace{-.02in}|\hspace{-.02in}|}}
\def\E{{\mathcal{E}}}
\def\T{{\mathcal{T}}}

\def\bpsi{\boldsymbol{\psi}}

\def\pT{{\partial T}}

\def\W{{\mathcal{W}}}

\def\bn{{\mathbf{n}}}
\def\bb{{\mathbf{b}}}

\def\bF{{\mathbf{F}}}

\newtheorem{remark}{Remark}[section]
\newtheorem{algorithm}{Primal-Dual Weak Galerkin Algorithm}[section]

\title {New Primal-Dual Weak Galerkin Finite Element Methods  for Convection-Diffusion Problems}

\begin{document}

\author{
Waixiang Cao \thanks{ School of Mathematical Sciences, Beijing Normal University, Beijing 100875, China (caowx@bnu.edu.cn). The research of Waixiang Cao was partially supported by NSFC grant No. 11871106.}
\and
Chunmei Wang \thanks{Department of Mathematics \& Statistics, Texas Tech University, Lubbock, TX 79409, USA (chunmei.wang@ttu.edu). The research of Chunmei Wang was partially supported by National Science Foundation Award DMS-1849483.}}

\maketitle

\begin{abstract}
This article devises a new primal-dual weak Galerkin finite element method for the convection-diffusion equation. Optimal order error estimates are established for the primal-dual weak Galerkin approximations in various discrete norms and the standard $L^2$ norms. A series of numerical experiments are conducted and reported to verify the theoretical findings.
\end{abstract}



\pagestyle{myheadings}

\section{Introduction}
This paper is concerned with new development of numerical methods for the convection-diffusion equations. For simplicity, we consider the model problem that seeks an unknown function $u$ satisfying
\begin{equation}\label{model}
\begin{split}
-\nabla\cdot (a \nabla u +\bb  u)=&f, \qquad \text{in}\quad \Omega,\\
u=&g_1, \qquad  \text{on}\quad  \Gamma_D,\\
(a \nabla u +\bb u)\cdot \bn=&g_2,  \qquad  \text{on}\quad   \Gamma_N,
\end{split}
\end{equation}
where $\Omega\subset \mathbb R^d (d=2, 3)$ is an open bounded polygonal ($d=2$) or polyhedral ($d=3$) domain with Lipschitz continuous boundary $\partial \Omega$, $\Gamma_D$ is the Dirichlet boundary, $\Gamma_N=\partial \Omega\setminus \Gamma_D$ is the Neumann boundary, and $\bn$ is the unit outward normal direction to the Neumann boundary $\Gamma_N$. We assume that the convection tensor $\bb\in [L^{\infty} (\Omega)]^d$ is bounded, and the diffusion tensor $a=\{a_{ij}\}_{d\times d}$ is symmetric and positive definite in the sense that there exists a constant $\alpha>0$, such that
$$
\xi^T a\xi\ge \alpha\xi^T\xi, \qquad \forall \xi\in \mathbb R^d.
$$
Furthermore, we assume that the diffusion tensor $a$ and the convection tensor $\bb$ are uniformly piecewise continuous functions.

The convection-diffusion equations arise in many areas of science and engineering. Readers are referred to the ``Introduction'' Section in \cite{wz2019} and the references cited therein for a detailed description of the convection-diffusion equations.

The weak Galerkin (WG) finite element method was first introduced by Wang and Ye in \cite{Wang:YeWG2013} for second order elliptic equations, and later was widely used for solving various partial differential equations, e.g., \cite{mu-wang-ye-2015, wang-ye-2014,wang-ye-2015, w1, w2, w3, w4, w5, w6, w7}. Recently, the authors in \cite{WW-mathcomp} have developed a new numerical scheme, called ``primal-dual weak Galerkin (PDWG) finite element method'' for the second order elliptic problem in non-divergence form. PDWG uses the weak Galerkin strategy to construct the {\it discrete weak Hessian} operator in the weak formulation of the model PDEs, and further seeks a discontinuous function which minimizes a stabilizer defined on the boundary of each element with the constraint given by the weak formulation of the model PDEs weakly defined on each element. The Euler-Lagrange method was employed to solve the constrained minimization problem leading to the primal-dual weak Galerkin finite element method, which has been further studied in \cite{ww2017, ww2018, wz2019, wwhyperbolic, w2018}. The primal-dual weak Galerkin finite element method has shown the promising features as a discretization approach due to: (1) it works well for a wide class of PDE problems for which no traditional variational formulations are available; and (2) it is applicable to virtually any PDE problems where the inf-sup condition is satisfied.

Using the usual integration by parts one may derive a weak formulation for the model problem (\ref{model}) as follows: Find  $u\in H^1(\Omega)$ satisfying $u|_{\Gamma_D}=g_1$ and $(a \nabla u+\bb u)\cdot \bn|_{\Gamma_N}=g_2$, such that  
\begin{equation}\label{weakform}
\begin{split}
&\int_T (a\nabla u+\bb  u)\cdot \nabla w dT -\int_{\partial T} (a\nabla u+\bb u) \cdot \bn w ds\\
=&\int_T fw dT, \quad \forall T\subset \Omega, \ w\in H^1(T).
\end{split}
\end{equation}
 The PDWG numerical scheme developed in this paper is based on the weak formulation \eqref{weakform} for the convection-diffusion model problem (\ref{model}). The gradient operator is the principal player in \eqref{weakform} so that a reconstructed gradient (i.e., weak gradient) is crucial in the PDWG finite element scheme. In contrast, the PDWG finite element method developed in \cite{wz2019} was based on a weak form principled by the operator ${\cal L}=\nabla \cdot (a\nabla)$ so that a reconstructed weak ${\cal L}$ played a key role in the construction of the numerical scheme. The two numerical methods are thus sharply different from each other, and each has its own advantage in theory and practical computation.

The rest of the paper is organized as follows. In Section 2, we present our primal-dual weak Galerkin scheme for the model problem \eqref{model} based on the weak formulation (\ref{weakform}). In Section 3, we shall establish a result on the solution existence and uniqueness for the numerical method. Section 4 is devoted to the establishment of the property of mass conservation. The error equations for the primal-dual weak Galerkin algorithm are derived in Section 5. Sections 6-7 are devoted to the establishment of some optimal order error estimates for the PDWG solution in discrete norms as well as the usual $L^2$-norm. Finally,  various numerical examples are presented in the last section to support our theoretical findings.

{\color{black} 

   Throughout this paper,  we adopt standard notations for Sobolev spaces such as $W^{m,p}(D)$ on sub-domain $D\subset\Omega$ equipped with
    the norm $\|\cdot\|_{m,p,D}$ and the semi-norm $|\cdot|_{m,p,D}$. When $D=\Omega$, we omit the index $D$; and if $p=2$, we set
   $W^{m,p}(D)=H^m(D)$,
   $\|\cdot\|_{m,p,D}=\|\cdot\|_{m,D}$, and $|\cdot|_{m,p,D}=|\cdot|_{m,D}$, and if $m=0,p=2$, we set $\|\cdot\|_{m,p,D}=\|\cdot\|_{D}$. 

}

\section{Numerical Algorithm}\label{Section:WGFEM}

Let ${\cal T}_h$ be a partition of the domain $\Omega$ into polygons in 2D or polyhedra in 3D which is shape regular in the sense of \cite{wy3655}. Denote by ${\mathcal E}_h$ the set of all edges or flat faces in ${\cal T}_h$ and  ${\mathcal E}_h^0={\mathcal E}_h \setminus \partial\Omega$ the set of all interior edges or flat faces. Denote by $h_T$ the meshsize of $T\in {\cal T}_h$ and
$h=\max_{T\in {\cal T}_h}h_T$ the meshsize for the partition ${\cal T}_h$.

By a weak function on $T\in\T_h$ we mean a triplet $v=\{v_0,v_b, v_n\}$ such that $v_0\in L^2(T)$, $v_b\in L^{2}(\partial T)$ and $v_n\in L^{2}(\partial T)$, where $\pT$ is the boundary of $T$. The first and the second components, namely $v_0$ and $v_b$, should be understood as the value of $v$ in the interior and on the boundary of $T$ respectively. The third component $v_n$ refers to the value of $(a\nabla v+\bb v) \cdot \bn$ on $\pT$. Note that $v_b$ and $v_n$ may not necessarily be the trace of $v_0$ and $(a\nabla v_0+\bb v_0) \cdot \bn$ on $\partial T$. Denote by $\W(T)$ the space of all weak functions on $T$; i.e.,
\begin{equation}\label{2.1}
\W(T)=\{v=\{v_0,v_b, v_n \}: v_0\in L^2(T), v_b\in
L^{2}(\partial T), v_n\in L^{2}(\partial T)\}.
\end{equation}

The weak gradient of $v\in \W(T)$, denoted by $\nabla_w  v$, is
defined as a linear functional on $[H^1(T)]^d$ such that
\begin{equation*}
(\nabla_w v,\boldsymbol{\psi})_T=-(v_0,\nabla \cdot
\boldsymbol{\psi})_T+\langle v_b,\boldsymbol{\psi}\cdot
\textbf{n}\rangle_{\partial T},
\end{equation*}
for all $\boldsymbol{\psi}\in [H^1(T)]^d$. Denote by $P_r(T)$ the space of polynomials on $T$ with degree $r\ge 0$. A discrete version of $\nabla_{w} v$, denoted by $\nabla_{w, r, T}v$, is defined as the unique vector-valued polynomial in $[P_r(T) ]^d$ satisfying
\begin{equation}\label{disgradient}
(\nabla_{w,r,T}  v, \boldsymbol{\psi})_T=-(v_0,\nabla \cdot
\boldsymbol{\psi})_T+\langle v_b, \boldsymbol{\psi} \cdot
\textbf{n}\rangle_{\partial T}, \quad\forall\boldsymbol{\psi}\in
[P_r(T)]^d.
\end{equation}
For smooth $v_0$, we have from the usual integration by parts that
\begin{equation}\label{disgradient*}
(\nabla_{w, r, T} v, \boldsymbol{\psi})_T= (\nabla v_0,
\boldsymbol{\psi})_T-\langle v_0- v_b, \boldsymbol{\psi} \cdot
\textbf{n}\rangle_{\partial T}, \quad\forall\boldsymbol{\psi}\in
[P_r(T)]^d.
\end{equation}


For any given integer $k\geq 1$, denote by $W_k(T)$ the local discrete weak function space; i.e.,
$$
W_k(T)=\{\{v_0, v_b, v_n\}: v_0\in P_k(T), v_b\in P_k(e),  v_n \in P_{l}(e),e\subset \partial T\},
$$
where $l=k-1$ or $l=k$. Patching $W_k(T)$ over all the elements $T\in {\cal T}_h$
through a common value $v_b$ and $\pm v_n$ on the interior interface $\E_h^0$, we arrive at a global
weak finite element space $W_h$; i.e.,
$$
W_h=\big\{\{v_0, v_b, v_n\}:\{v_0, v_b, v_n\}|_T\in W_k(T), \forall T\in {\cal T}_h \big\}.
$$
Denote by $W_h^0$ the subspace of $W_h$ with homogeneous Dirichlet  and Neumann boundary conditions; i.e.,
\begin{equation}\label{EQ:2020-Jan-17-000}
W_h^0=\{\{v_0, v_b, v_n\}\in W_h: v_b=0 \ \text{on}\ \Gamma_D, v_n=0 \ \text{on}\ \Gamma_N\}.
\end{equation}

Next, let $M_h$ be the finite element space consisting of piecewise polynomials of degree $k$; i.e.,
\begin{equation}\label{EQ:2020-Jan-17-001}
M_h=\{\sigma: \sigma|_T\in P_{k}(T),  \forall T\in {\cal T}_h\}.
\end{equation}

\begin{remark}
The finite element space $M_h$ in \eqref{EQ:2020-Jan-17-001} can also be constructed by using piecewise polynomials of degree $k-1$ in the forthcoming numerical scheme. All the mathematical results to be presented in this paper can be extended to the case of $k-1$ without any difficulty. 
\end{remark}

For simplicity,  for any $v=\{v_0, v_b, v_n\}\in
W_h$, denote by $\nabla_{w}v$ the discrete weak gradient
$\nabla_{w, k-1, T}v$ computed  by using
(\ref{disgradient}) on each element $T$; i.e.,
$$
(\nabla_{w}v)|_T= \nabla_{w ,k-1,T}(v|_T), \qquad v\in W_h.
$$
Let us introduce the following bilinear forms:
\begin{equation*}
\begin{split}
s(u, v)=&\sum_{T\in {\cal T}_h} h_T^{-3}\langle u_0-u_b, v_0-v_b\rangle_{\partial T}\\
&+ h_T^{-1} \langle (a \nabla u_0+\bb u_0)\cdot \bn-u_n, (a \nabla v_0+\bb v_0)\cdot \bn-v_n\rangle_{\partial T},\\
b(u,  \lambda)=&\sum_{T\in {\cal T}_h}(a\nabla_w u+\bb  u_0, \nabla  \lambda)_T-\langle u_n,  \lambda\rangle_{\partial T},\\
c(\lambda, \sigma)=& \tau_1\sum_{T\in {\cal T}_h} h_T^2 (\nabla \lambda, \nabla \sigma)_T+ \tau_2\sum_{T\in {\cal T}_h} h_T^4\sum_{i, j=1}^d(\partial_{ij}^2 \lambda,\partial_{ij}^2 \sigma)_T,
\end{split}
\end{equation*}
where $u, v\in W_h$ and $ \lambda, \sigma\in M_h$,  $\tau_1 \geq 0$ and $\tau_2 \geq 0$ are two mesh-independent parameters.

Let $k \geq 1$ and $T\in\T_h$. Denote by $Q_0^{(k)}$ the $L^2$ projection operator onto $P_k(T)$. For each edge or face $e\subset\partial T$, denote by $Q_b^{(k)}$ and $Q_n^{(l)}$ the $L^2$ projection operators onto $P_{k}(e)$ and $P_{l}(e)$, respectively. For any $w\in H^1(\Omega)$, denote by $Q_h w$ the $L^2$  projection onto the weak finite element space $W_h$ such that on each element $T$,
$$
Q_hw=\{Q_0^{(k)}w,Q_b^{(k)}w, Q_n^{(l)} ( (a\nabla w+\bb w) \cdot \bn)\}.
$$
Denote by ${\cal Q}^{(k-1)}_h$ the $L^2$ projection operator onto the space $[P_{k-1}(T)]^d$.

The numerical scheme for the convection-diffusion problem (\ref{model}) based on the variational formulation (\ref{weakform}) can be stated as follows:
\begin{algorithm}
Find $(u_h;\lambda_h)\in W_h
\times M_{h }$ satisfying $u_b=Q_b^{(k)} g_1$ on $\Gamma_D$ and  $u_n =Q_n ^{(l)} g_2$ on $\Gamma_N$, such that
\begin{eqnarray}\label{32}
s(u_h, v)+b(v, \lambda_h)&=& 0, \qquad \forall v\in W^0_{h},\\
-c(\lambda_h,  \sigma)+ b(u_h, \sigma)&=&(f,\sigma),\quad  \forall \sigma\in M_h.\label{2}
\end{eqnarray}
\end{algorithm}

\begin{remark}
For the case of $l=k$, one may take $\tau_1=\tau_2=0$ and thus $c(\lambda_h,  \sigma)=0$; for the case of $l=k-1$ and $k=1$, one may take $\tau_1>0$ and $\tau_2=0$; for the case of $l=k-1$ and $k\geq 2$, one would take $\tau_1=0$ and $\tau_2>0$, as suggested by the mathematical theory.
\end{remark}

\section{Solution Existence and Uniqueness}\label{Section:EU}
For the sake of analysis, in what follows of this paper, we assume that the diffusion tensor $a$ and the convection tensor $\bb$ in the convection-diffusion equation (\ref{model}) are piecewise constants in $\Omega$ with respect to the finite element partition ${\cal T}_h$. However, the analysis can be extended to the case that $a$ and $\bb$ are piecewise smooth functions without any difficulty. 

The $L^2$ projection operators $Q_h$ and ${\cal Q}^{(k-1)}_h$ satisfy the following commutative property \cite{wy3655}:
\begin{equation}\label{l}
\nabla_{w} (Q_h w) = {\cal Q}_h^{(k-1)}( \nabla w), \qquad \forall  w\in H^1(T).
\end{equation}
In the finite element spaces $W_h$ and $M_h$, we introduce the following seminorms:
\begin{align}\label{norm1}
\3bar v \3bar_{W_h} =& s(v, v)^{\frac{1}{2}},  \quad v \in W_h;\\ 
\3bar \sigma  \3bar_{M_h} =&c(\sigma, \sigma)^{\frac{1}{2}},  \quad \sigma \in M_h.\label{norm2}
\end{align}

\begin{lemma}\label{infsup}
(Generalized inf-sup condition)  For any $\lambda \in M_h$, there exists a $v\in W_h^0$ satisfying
\begin{equation}\label{infsup1}
b(v, \lambda) \geq \left\{\begin{split}
\frac{1}{2} \|\lambda\|^2, \qquad &   l=k,\\
\frac{1}{2} \|\lambda\|^2- \beta h^{2}\|\nabla\lambda\|^2, \quad &   k=1, l=k-1, \\
\frac{1}{2} \|\lambda\|^2- \beta h^{4}| \lambda |_2^2, \quad &    k\geq 2, l=k-1,
\end{split}\right.
 \end{equation}
\end{lemma}
for some constant $\beta>0$.
\begin{proof}
Consider the auxiliary problem of seeking $w$ such that
\begin{equation}\label{au1}
\begin{split}
-\nabla \cdot (a \nabla w+ \bb w)& =\lambda, \qquad \text{in}\ \Omega,\\
w&=0, \qquad \text{on}\ \Gamma_D,\\
(a\nabla w+\bb w) \cdot \bn&=0, \qquad \text{on}\ \Gamma_N.\\
\end{split}
\end{equation}
Assume that the auxiliary problem (\ref{au1}) has the  $ {\color{black} H^{2}}$-regularity property in the sense that there exists a constant $C$ satisfying
\begin{equation}\label{regu}
{\color{black}  \|w\|_{2}} \leq C\|\lambda\|.
\end{equation}
By taking $v=Q_hw=\{Q_0^{(k)}w, Q_b^{(k)} w, Q_n^{(l)} ((a\nabla w+\bb w)\cdot\bn)\}\in W_h^0$ in $b(v, \lambda)$, we have from (\ref{disgradient}) and the usual integration by parts that
\begin{equation}\label{b-term}
\begin{split}
&b(v, \lambda)= b(Q_hw, \lambda)\\
=&\sum_{T\in {\cal T}_h} (a \nabla_w Q_h w+\bb Q_0^{(k)} w, \nabla \lambda)_T- \langle Q_n^{(l)}((a\nabla w+\bb w)\cdot\bn), \lambda\rangle_{\partial T}\\
=&\sum_{T\in {\cal T}_h} (a {\cal Q}_h^{(k-1)}(\nabla w)+\bb Q_0^{(k)} w, \nabla \lambda)_T- \langle Q_n^{(l)}((a\nabla w+\bb w)\cdot\bn), \lambda\rangle_{\partial T}\\
=&\sum_{T\in {\cal T}_h} (a \nabla w+\bb w, \nabla \lambda)_T- \langle Q_n^{(l)}((a\nabla w+\bb w)\cdot\bn), \lambda\rangle_{\partial T}\\
=&\sum_{T\in {\cal T}_h}-(\nabla \cdot (a \nabla w +\bb w), \lambda)_T-\langle (Q_n^{(l)}-I)((a\nabla w+\bb w)\cdot\bn), \lambda\rangle_{\partial T} \\
=&\|\lambda\|^2-\sum_{T\in {\cal T}_h}\langle (Q_n^{(l)}-I)((a\nabla w+\bb w)\cdot\bn), (I-Q_n^{(l)})\lambda\rangle_{\partial T},
\end{split}
\end{equation}
where we have used the first equation of (\ref{au1}), (\ref{l}), and the property of the $L^2$ projection $Q_n^{(l)}$.

We shall discuss the estimate of the term  $\sum_{T\in {\cal T}_h}\langle (Q_n^{(l)}-I)((a\nabla w+\bb w)\cdot\bn), (I-Q_n^{(l)})\lambda\rangle_{\partial T}$ in various situations. 
For the case of $l=k$, we have
$$\sum_{T\in {\cal T}_h}\langle (Q_n^{(l)}-I)((a\nabla w+\bb w)\cdot\bn), (I-Q_n^{(l)})\lambda\rangle_{\partial T}=0,$$
which, together with \eqref{b-term}, gives (\ref{infsup1}) for the case of $l=k$. 
For the case of $l=k-1$, using the Cauchy-Schwarz inequality and the trace inequality \eqref{trace-inequality} gives
\begin{equation}\label{qnterm}
\begin{split}
&\left|\sum_{T\in {\cal T}_h}\langle (Q_n^{(l)}-I)((a\nabla w+\bb w)\cdot\bn), (I-Q_n^{(l)})\lambda\rangle_{\partial T}\right|\\
\leq & \Big(\sum_{T\in {\cal T}_h}\|(Q_n^{(l)}-I)((a\nabla w+\bb w)\cdot\bn)\|_{\pT}^2\Big)^{\frac{1}{2}} \Big(\sum_{T\in {\cal T}_h}\|(I-Q_n^{(l)})\lambda\|_{\pT}^2\Big)^{\frac{1}{2}}\\
\leq &C \Big(\sum_{T\in {\cal T}_h}h_T^{-1}\|(Q_0^{(l)}-I)((a\nabla w+\bb w))\|_{T}^2\\& +h_T\|(Q_0^{(l)}-I)(a\nabla w+\bb w)\|^2_{1,T} \Big)^{\frac{1}{2}}\\
& \Big(\sum_{T\in {\cal T}_h}h_T^{-1}\|(I-Q_0^{(l)})\lambda\|_{T}^2+h_T \|(I-Q_0^{(l)})\lambda\|_{1, T}^2\Big)^{\frac{1}{2}}\\
 \leq &
\left\{\begin{split}
  Ch\|\nabla\lambda\|  \|w\|_{2},\quad &  k=1, l=k-1,\\
 Ch^{2}|\lambda|_2 \|w\|_{2},\quad &  k\geq 2, l=k-1.
 \end{split}\right.
\end{split}
\end{equation}
Substituting (\ref{qnterm}) into (\ref{b-term}) and using  the Young's inequality and the $H^{2}$- regularity property \eqref{regu} gives
\begin{equation*}\label{qnterm-2}
\begin{split}
|b(v, \lambda)| 
& \geq \|\lambda\|^2-\epsilon \|w\|_{2}^2 -C\epsilon^{-1}\left\{\begin{split}
h^{2}\|\nabla\lambda\|^2 ,\quad &  k=1, l=k-1\\
h^{4}| \lambda |_2^2,\quad & k\geq 2, l=k-1
\end{split}\right.\\ \\
& \geq (1- \epsilon C) \|\lambda\|^2-C\epsilon^{-1}
\left\{\begin{split}
  h^{2}\|\nabla\lambda\|^2,\quad & k=1, l=k-1\\
h^{4}| \lambda |_2^2,\quad &  k\geq 2, l=k-1
\end{split}\right.\\
& \geq \frac{1}{2} \|\lambda\|^2-\beta  \left\{\begin{split}
 h^{2}\|\nabla\lambda\|^2,\quad & k=1, l=k-1,\\
h^{4}| \lambda |_2^2,\quad & k\geq 2, l=k-1,
\end{split}\right.
\end{split}
\end{equation*}
where   $\epsilon>0$ is a parameter  satisfying $1- \epsilon C \geq \frac{1}{2}$, and $\beta=C\epsilon^{-1} >0$.  
This completes the proof of (\ref{infsup1}) for the case of $l=k-1$ and further completes the proof of the lemma.
\end{proof}

\begin{theorem}\label{thmunique1} 
The primal-dual weak Galerkin algorithm (\ref{32})-(\ref{2}) has a unique solution.
\end{theorem}
\begin{proof}
It sufficies to prove that the homogeneous problem of  (\ref{32})-(\ref{2}) has only trivial solution. To this end, we assume $f=0$, $g_1=0$ and $g_2=0$. By letting $v=u_h$ and $\sigma=\lambda_h$ in (\ref{32})-(\ref{2}), we have from the difference of (\ref{32})-(\ref{2}) that
$$
s(u_h,u_h)+c (\lambda_h, \lambda_h)=0,
$$
which implies $u_0=u_b$ and $(a\nabla u_0+\bb u_0)\cdot \bn=u_n$ on each $\partial T$; and $c (\lambda_h, \lambda_h)=0$. From $c (\lambda_h, \lambda_h)=0$ we have $\nabla \lambda_h=0$ on each element $T\in {\cal T}_h$ if $\tau_1>0$ and $\partial^2_{ij}\lambda_h=0$ for $i, j=1, \cdots, d$ on each element $T\in {\cal T}_h$ if $\tau_2>0$, which shows that $c(\lambda_h, \sigma)= 0$ for all $\sigma \in M_h$.

Using (\ref{2}), (\ref{disgradient*}) and the usual integration by parts, we have
\begin{equation*}
\begin{split}
0=&b(u_h, \sigma)\\
=&\sum_{T\in {\cal T}_h}(a\nabla_w u_h+\bb  u_0, \nabla \sigma)_T-\langle u_n, \sigma\rangle_{\partial T}\\
=&\sum_{T\in {\cal T}_h} (\nabla  u_0,   a \nabla \sigma)_T- \langle u_0-u_b, a \nabla \sigma   \cdot  \bn \rangle_{\partial T}-(\nabla \cdot (\bb u_0), \sigma)_T  \\
&+\langle \bb u_0\cdot \bn, \sigma \rangle_{\partial T} -\langle u_n, \sigma\rangle_{\partial T} \\ =&\sum_{T\in {\cal T}_h} -( \nabla \cdot( a\nabla  u_0),    \sigma)_T+ \langle  a \nabla u_0 \cdot  \bn , \sigma\rangle_{\partial T}- \langle u_0-u_b, a \nabla \sigma   \cdot  \bn \rangle_{\partial T}   \\
&-(\nabla \cdot (\bb u_0), \sigma)_T+\langle \bb u_0\cdot \bn, \sigma \rangle_{\partial T} -\langle u_n, \sigma\rangle_{\partial T} \\ =&\sum_{T\in {\cal T}_h}- ( \nabla \cdot(a \nabla u_0+\bb u_0), \sigma)_T  - \langle u_0-u_b, a \nabla \sigma   \cdot  \bn \rangle_{\partial T}\\
&+\langle (a\nabla u_0+\bb u_0)\cdot \bn-u_n, \sigma\rangle_{\partial T} \\
=&\sum_{T\in {\cal T}_h} -(\nabla \cdot (a\nabla u_0+\bb u_0),  \sigma)_T,\\
\end{split}
\end{equation*}
where we used $u_0=u_b$ and $(a\nabla u_0+\bb u_0)\cdot \bn=u_n$ on each $\partial T$.
This gives $\nabla \cdot (a\nabla u_0+\bb u_0)=0$ on each element $T\in {\cal T}_h$ by taking $\sigma=\nabla \cdot (a\nabla u_0+\bb u_0)$. From $(a\nabla u_0+\bb u_0)\cdot \bn=u_n$ on each $\partial T$, and $a\nabla u_0+\bb u_0 \in H(div; T)$, we obtain $a\nabla u_0+\bb u_0 \in H(div; \Omega)$ and further $\nabla \cdot (a\nabla u_0+\bb u_0)=0$ in $\Omega$. Using $g_1=0$ on $\Gamma_D$ and $u_0=u_b$ on each $\partial T$, gives $u_0=0$ on $\Gamma_D$. Using $g_2=0$ on $\Gamma_N$ and $(a\nabla u_0+\bb u_0)\cdot \bn=u_n$ on each $\partial T$, yields $(a\nabla u_0+\bb u_0)\cdot \bn=0$ on $\Gamma_N$. Therefore, from the solution uniqueness of the PDE problem, we have $u_0\equiv 0$ in $\Omega$. We further obtain $u_b \equiv 0$, $u_n \equiv 0$ and thus $u_h \equiv 0$ in $\Omega$.

From $u_h \equiv 0$ in $\Omega$,    (\ref{32}) can be simplified as follows 
$$b(v, \lambda_h)=0,\qquad  \forall v\in W_h^0.$$
From Lemma \ref{infsup}, there exists a $v\in W_h^0$, satisfying
\begin{equation}\label{inf}
0=b(v, \lambda_h)\geq  \left\{\begin{split}
\frac{1}{2} \|\lambda_h\|^2,\qquad & l=k,\\
\frac{1}{2} \|\lambda_h\|^2- \beta h^{2}\|\nabla\lambda_h\|^2,\quad & k=1, l=k-1, \\
\frac{1}{2} \|\lambda_h\|^2- \beta h^{4}| \lambda_h|_2^2,\quad & k\geq 2, l=k-1,
\end{split}\right.
\end{equation}
for some constant $\beta>0$.  For the case of $l=k$,  it follows from (\ref{inf}) that $\lambda_h \equiv 0$ in $\Omega$. Note that when $l=k-1$ and $k=1$, we take $\tau_1>0$ and $\tau_2=0$; when $l=k-1$ and $k\geq 2$, we take $\tau_1=0$ and $\tau_2>0$. Thus, for the case of $l=k-1$, using $c(\lambda_h, \lambda_h) \equiv 0$  gives $\nabla \lambda_h=0$ on each $T\in {\cal T}_h$ for $k=1$; and $\partial_{ij}^2 \lambda_h=0$ for any $i, j=1, \cdots, d$ on each $T\in {\cal T}_h$ for $k\geq 2$, which, combined with  \eqref{inf},  yields $\lambda_h\equiv 0$ in $\Omega$ for the case of $l=k-1$. This completes the proof of this theorem.
\end{proof}

\section{Mass Conservation} \label{Section:MC}
The first equation in the convection-diffusion model problem (\ref{model}) can be rewritten in a conservative form; i.e.,
\begin{eqnarray} \label{eq1-1}
-\nabla \cdot \textbf{F}&=&f, \\
\label{eq2}
 \textbf{F}&=&a\nabla u+\bb u.
\end{eqnarray}
On each element $T\in \T_h$,  integrating (\ref{eq1-1}) over $T$ gives the integral formulation of the mass conservation; i.e.,
\begin{equation}\label{mas}
-\int_{\partial T}\textbf{F} \cdot \bn ds =\int_T fdT.
\end{equation}

We claim that the numerical solution arising from the primal-dual weak Galerkin scheme (\ref{32})-(\ref{2}) for the convection-diffusion  model problem (\ref{model}) retains the mass conservation property (\ref{mas}) locally on each element $T\in {\cal T}_h$ with a numerical flux $\bF_h$. To this end, for any given element $T\in {\cal T}_h$, choosing the test function  $\sigma$  in (\ref{2})  such that $\sigma=1$ on $T$ and $\sigma=0$ elsewhere,  yields
\begin{equation*}
\begin{split}
&\quad -\tau_1 h_T^2 (\nabla \lambda_h, \nabla 1)_T-\tau_2 h_T^4\sum_{i, j=1}^d(\partial_{ij}^2 \lambda_h, \partial_{ij}^2 1)_T+(a\nabla_w u_h+\bb u_0, \nabla 1)_T-\langle u_n, 1\rangle_{\pT}\\
&=(f, 1)_T,
\end{split}
\end{equation*}
which can be simplified as follows
\begin{equation*}\label{eq1-3}
-\langle u_n \bn \cdot \bn, 1\rangle_{\pT}=(f, 1)_T.
\end{equation*}
This implies that the primal-dual weak Galerkin algorithm (\ref{32})-(\ref{2}) conserves mass with a numerical flux given by
$$
\textbf{F}_h|_\pT=  u_n  \bn.
$$
It is easy to check that
$$
\bF_h|_{\pT_1} \cdot \bn_{T_1} + \bF_h|_{\pT_2} \cdot \bn_{T_2} = 0,\quad \mbox{on} \ e=\pT_1\cap\pT_2,
$$
where $\bn_{T_1}$ and $\bn_{T_2}$ are the unit outward normal directions along the interior edge or flat face $e=\partial T_1\cap \partial T_2$ pointing exterior to $T_1$ and $T_2$, respectively. This indicates the continuity of the numerical flux $\bF_h$ along the normal direction on each interior edge or flat face $e\in {\cal E}_h^0$.

The result can be summarized as follows.

\begin{theorem}\label{THM:conservation} Let $(u_h=\{u_0, u_b, u_n\};\lambda_h)$ be the numerical solution of the convection-diffusion model problem \eqref{model} arising from the primal-dual weak Galerkin finite element method (\ref{32})-(\ref{2}). Define a numerical flux function as follows:
\begin{eqnarray*}
\textbf{F}_h|_\pT :=u_n  \bn, \quad\mbox{on } \pT, \ T\in \T_h.
\end{eqnarray*}
Then, the numerical flux approximation $\textbf{F}_h$ is continuous across each interior edge or flat face  $e\in {\cal E}_h^0$ in the normal direction, and satisfies the following mass conservation property; i.e.,
\begin{equation*}
-\int_{\partial T}\textbf{F}_h \cdot \bn ds =\int_T fdT.
\end{equation*}
\end{theorem}

\section{Error Equations}
Let $u$ and $(u_h;\lambda_h) \in W_h\times M_h$ be the exact solution of (\ref{model}) and the PDWG solution arising from the numerical scheme (\ref{32})-(\ref{2}), respectively. Denote by ${\cal Q}^{(k)}_h$ the $L^2$ projection onto the finite element space $M_h$. Note that the exact solution of  the Lagrange multiplier $\lambda$ is $0$.
Define two error functions by
\begin{align}\label{error}
e_h&=u_h-Q_hu,
\\
\epsilon_h&=\lambda_h-{\cal Q}^{(k)}_h\lambda=\lambda_h.\label{error-2}
\end{align}

\begin{lemma}\label{errorequa}
The error functions $e_h$
and $\epsilon_h$ defined in (\ref{error})-(\ref{error-2}) satisfy the following   error equations for the primal-dual WG finite element scheme
(\ref{32})-(\ref{2}); i.e.,
\begin{eqnarray}\label{sehv}
s(e_h , v)+b(v, \epsilon_h)&=& -s(Q_hu, v), \qquad \forall v\in W^0_{h},\\
-c(\epsilon_h, \sigma)+b(e_h, \sigma)&=&\ell_u(\sigma),\qquad\qquad \forall \sigma\in M_h, \label{sehv2}
\end{eqnarray}
\end{lemma}
where \begin{equation}\label{lu}
\qquad \ell_u(\sigma) =
\left\{\begin{split}
0,\quad &l=k,\\
\sum_{T\in {\cal T}_h} \langle (Q_n^{(l)}-I)((a\nabla u+\bb u)\cdot \bn), \sigma \rangle_{\partial T},\quad &l=k-1.
\end{split}\right.
\end{equation}

\begin{proof}
Note that the exact solution of the Lagrange multiplier $\lambda$ is $0$. Subtracting $s(Q_hu, v)$ from both sides of (\ref{32}) yields
\begin{equation*}
\begin{split}
&s(u_h-Q_hu, v)+b(v, \lambda_h-{\cal Q}^{(k)}_h\lambda)=-s(Q_hu, v),  \qquad \forall v\in W^0_h.
\end{split}
\end{equation*}
This completes the proof of \eqref{sehv}. Next, for any $\sigma\in M_h$, we have
\begin{equation*}
\begin{split}
b(Q_h u, \sigma)= &\sum_{T\in {\cal T}_h} (a\nabla_w Q_h u+\bb Q_0^{(k)} u, \nabla \sigma)
_T- \langle Q_n^{(l)}((a\nabla u+\bb u)\cdot \bn), \sigma \rangle_{\partial T}\\
= &\sum_{T\in {\cal T}_h} (a {\cal Q}_h^{(k-1)} \nabla u+\bb Q_0^{(k)} u, \nabla \sigma)
_T- \langle Q_n^{(l)}((a\nabla u+\bb u)\cdot \bn), \sigma \rangle_{\partial T}\\
= &\sum_{T\in {\cal T}_h} (a   \nabla u+\bb  u, \nabla \sigma)
_T- \langle Q_n^{(l)}((a\nabla u+\bb u)\cdot \bn), \sigma \rangle_{\partial T}\\
= &\sum_{T\in {\cal T}_h} -(\nabla  \cdot (a\nabla u+\bb  u),  \sigma)_T +\langle  (a\nabla u+\bb u)\cdot \bn, \sigma \rangle_{\partial T}  \\
& - \langle Q_n^{(l)}((a\nabla u+\bb u)\cdot \bn), \sigma \rangle_{\partial T}\\
 = &\sum_{T\in {\cal T}_h}  (f,  \sigma)_T- \sum_{T\in {\cal T}_h}\langle (Q_n^{(l)}-I)((a\nabla u+\bb u)\cdot \bn), \sigma \rangle_{\partial T},\end{split}
\end{equation*}
where we have used the operator identify (\ref{l}), the usual integration by parts, and the first equation of (\ref{model}). Note that for the case of $l=k$, we have $\sum_{T\in {\cal T}_h}\langle (Q_n^{(l)}-I)((a\nabla u+\bb u)\cdot \bn), \sigma \rangle_{\partial T}=0$. Combining the above with (\ref{2}) yields (\ref{sehv2}). This completes the proof of the lemma.
\end{proof}

\section{Residual Error Estimates}
Recall that $\T_h$ is a shape-regular finite element partition of
the domain $\Omega$. For any $T\in\T_h$ and $\varphi\in H^1(T)$, the
following trace inequality holds true \cite{wy3655}:
\begin{equation}\label{trace-inequality}
\|\varphi\|_{\pT}^2 \leq C
(h_T^{-1}\|\varphi\|_{T}^2+h_T\|\nabla\varphi\|_{T}^2).
\end{equation}
If $\varphi$ is a polynomial on the element $T\in \T_h$, then from
the inverse inequality (see also \cite{wy3655}) we have
\begin{equation}\label{x}
\|\varphi\|_{\pT}^2 \leq C h_T^{-1}\|\varphi\|_{T}^2.
\end{equation}

\begin{lemma}\label{Lemma5.2}\cite{wy3655}  Let ${\cal T}_h$ be a
finite element partition of the domain $\Omega$ satisfying the shape regularity
assumptions given in \cite{wy3655}. Then, for any $0\leq p\leq 2$, $1\leq m\leq k$, one has
\begin{eqnarray}\label{3.2}
\sum_{T\in {\cal T}_h}h_T^{2p}\|u-Q_0^{(m)}u\|^2_{p,T} &\leq &
Ch^{2(m+1)}\|u\|_{m+1}^2,\\
  \sum_{T\in {\cal
T}_h} h_T^{2p}\|\nabla u-{\cal Q}^{(m-1)}_h \nabla u\|^2_{p,T} &\leq& Ch^{2m}\|u\|_{m+1}^2, 
 \label{3.3} \\
 \sum_{T\in {\cal T}_h}h_T^{2p}\|u-{\cal Q}^{(m)}_hu\|^2_{p,T} &\leq &
Ch^{2(m+1)}\|u\|_{m+1}^2. \label{3.3-2}
\end{eqnarray}
\end{lemma}

\begin{theorem} \label{theoestimate} 
Let $u$ and $(u_h;\lambda_h) \in
W_{h}\times M_{h}$ be the exact solution of (\ref{model}) and PDWG solution of
(\ref{32})-(\ref{2}), respectively. Assume that the exact solution
$u$ of (\ref{model}) is sufficiently regular such that $u\in
H^{k+1}(\Omega)$. Then, there exists a constant $C$ such that the following error estimate holds true:
\begin{equation}\label{erres}
\3bar u_h-Q_h u \3bar_{W_h}+\3bar \lambda_h-{\cal Q}^{(k)}_h \lambda\3bar_{M_h} \leq
 \left\{\begin{split}
C h^{k-1}\|u\|_{k+1},   & \  l=k,\\
C(1+\tau_1^{-\frac{1}{2}})h^{k-1}\|u\|_{k+1}, & \ k=1, l=k-1, \\
C(1+\tau_2^{-\frac{1}{2}}) h^{k-1}\|u\|_{k+1},  & \ k\geq 2, l=k-1.
\end{split}\right.
\end{equation}
\end{theorem}

\begin{proof} By choosing $v=e_h$ and $\sigma=\epsilon_h$ in (\ref{sehv})-(\ref{sehv2}), we have from the difference of (\ref{sehv}) and (\ref{sehv2}) that
\begin{equation}\label{ess2}
s(e_h, e_h)+c (\epsilon_h, \epsilon_h)=-s(Q_h u, e_h)-\ell_u(\epsilon_h).
\end{equation}

Recall that
\begin{equation}\label{EQ:September:03:100}
\begin{split}
&s(Q_h u, e_h)\\
= &\sum_{T\in {\cal T}_h}h_T^{-3}\langle Q_0^{(k)}u-Q_b^{(k)}u,
e_0-e_b\rangle_\pT+\sum_{T\in {\cal T}_h}h_T^{-1} \langle (a\nabla Q_0^{(k)}u+\bb  Q_0^{(k)}u) \cdot \bn \\ & -Q_n^{(l)}((a\nabla u+\bb u)  \cdot \bn), (a\nabla e_0+\bb e_0) \cdot \bn - e_n\rangle_\pT.
\end{split}
\end{equation}
The first term on the right-hand side of (\ref{EQ:September:03:100})
can be estimated by using the Cauchy-Schwarz inequality, the trace
inequality (\ref{trace-inequality}), and the estimate (\ref{3.2}) with $m=k$ as follows
\begin{equation}\label{s1}
\begin{split}
&\left| \sum_{T\in {\cal T}_h}h_T^{-3}\langle Q_0^{(k)}u-Q_b^{(k)}u, e_0-e_b\rangle_\pT\right|\\
= &\left| \sum_{T\in {\cal T}_h}h_T^{-3}\langle Q_0^{(k)}u- u, e_0-e_b\rangle_\pT\right|\\
 \leq & \Big(\sum_{T\in {\cal T}_h}h_T^{-3}\|u-Q_0^{(k)}u\|^2_{\partial
T}\Big)^{\frac{1}{2}} \Big(\sum_{T\in {\cal T}_h}
h_T^{-3}\|e_0-e_b\|^2_{\partial T}\Big)^{\frac{1}{2}}\\
\leq & C\Big(\sum_{T\in {\cal
T}_h}h_T^{-4} \|u-Q_0^{(k)}u\|_T^2+h_T^{-2}|u-Q_0^{(k)}u|_{1,T}^2
\Big)^{\frac{1}{2}}\3bar e_h \3bar_{W_h}\\
\leq & Ch^{k-1}\|u\|_{k+1}\3bar e_h \3bar_{W_h}.
\end{split}
\end{equation}
Similarly, the second term on the right-hand side of
(\ref{EQ:September:03:100}) has the following estimate
\begin{equation}\label{s2}
\begin{split}
&\qquad \Big|\sum_{T\in {\cal T}_h}h_T^{-1} \langle (a\nabla Q_0^{(k)}u+\bb  Q_0^{(k)}u) \cdot \bn  -Q_n^{(l)}((a\nabla u+\bb u) \cdot \bn), \\&(a\nabla e_0+\bb e_0) \cdot \bn  -e_n\rangle_\pT
\Big|
\leq Ch^{k-1}\|u\|_{k+1} \3bar e_h \3bar_{W_h}.
\end{split}
\end{equation}
Substituting (\ref{s1}) and (\ref{s2}) into (\ref{EQ:September:03:100})  
gives
\begin{equation}\label{EQ:True:01}
|s(Q_h u, e_h)| \leq Ch^{k-1}\|u\|_{k+1}\3bar e_h \3bar_{W_h}.
\end{equation}

We shall further discuss the second term on the right-hand side of (\ref{ess2}).
For the case of $l=k$, from (\ref{lu}), we have
\begin{equation}\label{te1}
\ell_u(\epsilon_h)=0.
\end{equation}
{\color{black} We now consider the case of $l=k-1$.  By denoting 
\[
   F_u=a\nabla u+\bb u, 
\]
and then using (\ref{lu}), the Cauchy-Schwarz inequality, the trace
inequality (\ref{trace-inequality}),  and the estimate (\ref{3.2}) with $m=l=k-1$, we have
\begin{equation}\label{aij}
\begin{split}
|\ell_u(\epsilon_h) | = & |\ell_u(\epsilon_h-I_h^l{\epsilon_h}) | 
= \left| \sum_{T\in {\cal T}_h}\langle (Q_n^{(l)}-I)(F_u\cdot\bn), \epsilon_h-I_h^l{\epsilon_h}\rangle_{\partial T}\right|\\
  \leq & \Big( \sum_{T\in {\cal T}_h} \|(Q_n^{(l)}-I)(F_u\cdot\bn)\|_{\pT}^2\Big)^{\frac{1}{2}}\Big( \sum_{T\in {\cal T}_h} \|  \epsilon_h-I_h^l{\epsilon_h}\|_{\pT}^2\Big)^{\frac{1}{2}} \\
\leq & C\Big(\sum_{T\in {\cal T}_h} h_T^{-1}\|(Q_0^{(l)}-I)F_u\|_{T}^2+h_T |(Q_0^{(l)}-I)F_u|_{1, T}^2\Big)^{\frac{1}{2}} \\
 & \ 
\Big( \sum_{T\in {\cal T}_h} h_T^{-1}\|\epsilon_h-I_h^l{\epsilon_h}\|_{T}^2+h_T\| \nabla(\epsilon_h-I_h^l{\epsilon_h})\|_{T}^2\Big)^{\frac{1}{2}} \\
\leq & Ch^{l} \|F_u\|_{l+1} \Big( \sum_{T\in {\cal T}_h} \|\epsilon_h-I_h^l{\epsilon_h}\|_{T}^2+h^2_T\| \nabla(\epsilon_h-I_h^l{\epsilon_h})\|_{T}^2\Big)^{\frac{1}{2}}, 
\end{split}
\end{equation}
where $I_h^l {\epsilon_h}$ denotes the cell average and linear interpolation of $\epsilon_h$
 on each element $T\in {\cal T}_h$ for  $l=0$ and $l\ge 1$, resepctively. 
  Choosing $l=k-1$ in the above inequality and using the approximation property of the  interpolation function yields }
\begin{equation}\label{aij-3}
|\ell_u(\epsilon_h) | \leq
 \left\{\begin{split}
C\tau_1^{-\frac{1}{2}} h^{k-1}\|u\|_{k+1} \3bar\epsilon_h\3bar_{M_h},\quad & k=1, l=k-1,
 \\
C\tau_2^{-\frac{1}{2}} h^{k-1}\|u\|_{k+1} \3bar\epsilon_h\3bar_{M_h},\quad &  k\geq 2, l=k-1.
  \end{split}\right.
\end{equation}

Substituting  (\ref{EQ:True:01}), (\ref{te1}),  and (\ref{aij-3}) into (\ref{ess2}) gives the error estimate (\ref{erres}). This completes the proof of the theorem.
\end{proof}

\begin{theorem} Under the assumption of Theorem \ref{theoestimate}, there exists a constant $C$ such that the following error estimate holds true:
\begin{equation}\label{rese0}
 \Big(\sum_{T\in {\cal T}_h} \|\nabla\cdot(a\nabla e_0 + \bb e_0)\|^2_T \Big)^{\frac{1}{2}}\\
 \leq
 \left\{\begin{split}
C h^{k-1}\|u\|_{k+1},   & \ l=k,\\
C(1+\tau_1^{\frac{1}{2}})(1+\tau_1^{-\frac{1}{2}})h^{k-1}\|u\|_{k+1}, & \ k=1, l=k-1, \\
C(1+\tau_2^{\frac{1}{2}})(1+\tau_2^{-\frac{1}{2}}) h^{k-1}\|u\|_{k+1},  & \ k\geq 2, l=k-1.
\end{split}\right.
\end{equation}
\end{theorem}
\begin{proof}
From the error equation \eqref{sehv2} we have
\begin{equation}\label{au1-1}
b(e_h, \sigma)=c(\epsilon_h, \sigma)+\ell_u(\sigma),\qquad \forall \sigma\in M_h.
\end{equation} 
Recall that
\begin{equation}\label{au2-1}
\begin{split}
b(e_h, \sigma) = & \sum_{T\in {\cal T}_h} (a\nabla_w e_h+\bb e_0 , \nabla\sigma)_T -\langle e_n, \sigma\rangle_\pT\\
= & \sum_{T\in {\cal T}_h} (a\nabla e_0+\bb e_0 , \nabla\sigma)_T + \langle e_b-e_0, a\nabla\sigma\cdot\bn\rangle_\pT-\langle e_n, \sigma\rangle_\pT \\
=&-\sum_{T\in {\cal T}_h} (\nabla\cdot(a\nabla e_0 + \bb e_0), \sigma)_T -\langle e_b-e_0, a\nabla\sigma\cdot\bn\rangle_\pT\\&+\langle e_n - (a\nabla e_0 + \bb e_0)\cdot\bn, \sigma\rangle_\pT, 
 \end{split}
\end{equation}
where we have used (\ref{disgradient*}) with $\bpsi=a \nabla\sigma$ and the usual integration by parts.
Substituting (\ref{au2-1}) into (\ref{au1-1}) gives
\begin{equation}\label{EQ:aaa:001}
\begin{split}
&-\sum_{T\in {\cal T}_h} (\nabla\cdot(a\nabla e_0 + \bb e_0), \sigma)_T \\
=&
c(\epsilon_h, \sigma)+\ell_u(\sigma) +\sum_{T\in {\cal T}_h}\langle e_0-e_b, a\nabla\sigma\cdot\bn\rangle_\pT +\langle e_n - (a\nabla e_0 + \bb e_0)\cdot\bn, \sigma\rangle_\pT\\
=&J_1+J_2+J_3+J_4,
\end{split}
\end{equation}
where $J_i $ is defined accordingly for $i=1,\cdots,4$.

We shall estimate each term $J_i$ in (\ref{EQ:aaa:001}) respectively. 
With $J_1$, we have for the case of $l=k$, $J_1=0$. For the case of $l=k-1$ and $k=1$, we have
\begin{equation*}
\begin{split}
J_1=&\tau_1 \sum_{T\in {\cal T}_h} h_T^2 (\nabla \epsilon_h, \nabla\sigma)_T\\
\leq &   \Big( \sum_{T\in {\cal T}_h} \tau_1 h_T^2 \|\nabla \epsilon_h\|_T^2\Big)^{\frac{1}{2}}\Big( \sum_{T\in {\cal T}_h}\tau_1 h_T^2 \|\nabla\sigma\|_T\Big)^{\frac{1}{2}}\\
\leq & C\tau_1^{\frac{1}{2}}\3bar \epsilon_h \3bar_{M_h} \Big( \sum_{T\in {\cal T}_h}   \| \sigma\|_T\Big)^{\frac{1}{2}},
\end{split}
\end{equation*}
where we have used the Cauchy-Schwarz inequality and the inverse inequality. Similarly, for the case of $l=k-1$ and $k\geq 2$, we have
$$
J_1  \leq C\tau_2^{\frac{1}{2}}\3bar \epsilon_h \3bar_{M_h} \Big( \sum_{T\in {\cal T}_h}   \| \sigma\|_T\Big)^{\frac{1}{2}}.$$ 

As to the term $J_2$, we have from (\ref{lu})  that for $l=k$, $\ell_u(\sigma)=0$; for $l=k-1$, we have, {\color{black} by following the same argument as that in \eqref{aij}
\begin{equation*} 
 \begin{split}
|J_2|=|\ell_u(\sigma)| 
 \leq & \Big(\sum_{T\in {\cal T}_h} \| (Q_n^{(l)}-I)((a\nabla u+\bb u)\cdot \bn)\|_{\partial T}^2  \Big)^{\frac{1}{2}}\Big(\sum_{T\in {\cal T}_h} \| \sigma\|_{\partial T} ^2  \Big)^{\frac{1}{2}}\\
  \leq & Ch^{k-1}\|u\|_{k+1}\Big(\sum_{T\in {\cal T}_h}  \| \sigma\|_{T} ^2  \Big)^{\frac{1}{2}},
\end{split} 
\end{equation*}
where we used the Cauchy-Schwarz inequality, the estimate (\ref{3.2}) with $m=l=k-1$, and the trace inequalities (\ref{trace-inequality}) and (\ref{x}).
}
As to the term $J_3$, we have
\begin{equation*} 
 \begin{split}
J_3  \leq  &   \Big(\sum_{T\in {\cal T}_h}  h_T^{-3}\|e_0-e_b\|^2_{\pT} \Big)^{\frac{1}{2}}\Big(\sum_{T\in {\cal T}_h} h_T^3\|a\nabla\sigma\cdot\bn\|^2_{\pT} \Big)^{\frac{1}{2}}\\
  \leq  & C  \3bar e_h \3bar_{W_h}\Big(\sum_{T\in {\cal T}_h} h_T^2\|a\nabla\sigma\cdot\bn\|^2_{T}  \Big)^{\frac{1}{2}}\\
   \leq  & C  \3bar e_h \3bar_{W_h}\Big(\sum_{T\in {\cal T}_h} \| \sigma \|^2_{T}  \Big)^{\frac{1}{2}},
\end{split} 
\end{equation*}
where we used the Cauchy-Schwarz inequality, the trace inequality (\ref{x}) and the inverse inequality.

 For the last term $J_4$, we have
 \begin{equation*} 
 \begin{split}
 &\sum_{T\in {\cal T}_h} \langle e_n - (a\nabla e_0 + \bb e_0)\cdot\bn, \sigma\rangle_\pT
\\
\leq  &   \Big(\sum_{T\in {\cal T}_h}  h_T^{-1}\|e_n - (a\nabla e_0 + \bb e_0)\cdot\bn\|^2_{\pT} \Big)^{\frac{1}{2}}\Big(\sum_{T\in {\cal T}_h} h_T  \|\sigma\|^2_{\pT} \Big)^{\frac{1}{2}}\\
\leq  & C  \3bar e_h \3bar_{W_h}\Big(\sum_{T\in {\cal T}_h}  \|\sigma\|^2_{T} \Big)^{\frac{1}{2}},
\end{split} 
\end{equation*}
where we used the Cauchy-Schwarz inequality and the trace inequality (\ref{x}).
 
Substituting the above estimates for $J_i (i=1, \cdots, 4)$ into (\ref{EQ:aaa:001}) and combining with (\ref{erres}) completes the proof of (\ref{rese0}).
\end{proof}

\section{Error Estimates in $H^1$ and $L^2$}\label{Section:H1L2Error}

Consider the dual problem of seeking an unknown function $w$ such that
\begin{align}\label{dual1}
-\nabla \cdot (a \nabla w) +\bb\cdot  \nabla w=& e_0,\qquad \text{in}\ \Omega,\\
w=& \ 0,\qquad \text{on}\ \Gamma_D, \label{dual2}\\
 a \nabla w \cdot\bn=& \ 0,\qquad \text{on}\ \Gamma_N, \label{dualN}
\end{align}
for any given $e_0\in L^2(\Omega)$.
The problem (\ref{dual1})-(\ref{dualN}) is said to be
$H^{1+s} (\frac{1}{2} <s\leq 1)$-regular in the sense that
\begin{equation}\label{regul}
\|w\|_{1+s}\leq C\|e_0\|.
\end{equation}

\begin{lemma} Let $e_h=\{e_0, e_b, e_n\}$ be the error function defined in (\ref{error}).  There holds
\begin{equation}\label{ehe0}
\|\nabla_w e_h-\nabla e_0\|_T \leq C h_T^{-\frac{1}{2}} \|e_0-e_b\|_{\pT}.
\end{equation}
\end{lemma}
\begin{proof}
From (\ref{disgradient*}), we have
$$
(\nabla_w e_h-\nabla e_0, \bpsi)_T=-\langle e_0-e_b, \bpsi \cdot \bn\rangle_{\pT}, \qquad \forall  \bpsi \in [P_{k-1}(T)]^d.
$$
From the Cauchy-Schwarz inequality and the trace inequality (\ref{x}), we thus have
\begin{equation*}
\begin{split}
\|\nabla_w e_h-\nabla e_0\|_T\leq &\sup_{\forall \bpsi \in [P_{k-1}(T)]^d}\frac{ \|e_0-e_b\|_{\pT}\|\bpsi \cdot \bn\|_{\pT}}{\|\bpsi\|_T}
\\ \leq & \sup_{\forall \bpsi \in [P_{k-1}(T)]^d}\frac{Ch_T^{-\frac{1}{2}} \|e_0-e_b\|_{\pT}\|\bpsi \|_{T}}{\|\bpsi\|_T}\\
\\ \leq & C h_T^{-\frac{1}{2}} \|e_0-e_b\|_{\pT}.
\end{split}
\end{equation*}
This completes the proof of the lemma.
\end{proof}

The following theorem presents the error estimate in the usual $L^2$ norm for the first component $u_0$ in the primal variable $u_h=\{u_0, u_b, u_n\}$ of the PDWG solution arising from the numerical scheme (\ref{32})-(\ref{2}).

\begin{theorem}\label{Thm:L2errorestimate} Assume that the dual problem (\ref{dual1})-(\ref{dualN}) has the $H^{1+s}$-regularity with a priori estimate (\ref{regul}) for $s\in (\frac{1}{2},1]$. There exists a constant $C$ such that
\begin{equation}\label{e0}
\|e_0\|  \leq  \left\{\begin{split}
C h^{k+s}\|u\|_{k+1},\quad &l=k,\\
C(1+\tau_1^{\frac{1}{2}}(1+\tau_1^{-\frac{1}{2}}))h^{k}\|u\|_{k+1},\quad & k=1, l=k-1,   \\
C(1+\tau_2^{\frac{1}{2}})(1+\tau_2^{-\frac{1}{2}})h^{k+s}\|u\|_{k+1},\quad &k\geq 2,  l=k-1.
  \end{split}\right.   
\end{equation}
provided that the meshsize $h$ is sufficiently small.
\end{theorem}

\begin{proof} Testing (\ref{dual1}) with $e_0$ on each element $T\in {\cal T}_h$, we obtain from the usual integration by parts that
\begin{equation}\label{2.14:800:10:L2}
\begin{split}
\|e_0\|^2 =&\sum_{T\in{\cal T}_h}(-\nabla \cdot (a \nabla w) +\bb\cdot  \nabla w, e_0)_T\\
 = & \sum_{T\in{\cal T}_h}(a\nabla w, \nabla e_0)_T-\langle a\nabla w\cdot \bn, e_0\rangle_{\partial T}+(\bb\cdot  \nabla w, e_0)_T\\
 = &  \sum_{T\in{\cal T}_h}(a\nabla w, \nabla e_0)_T-\langle a\nabla w\cdot \bn, e_0-e_b\rangle_{\partial T}+(\bb\cdot  \nabla w, e_0)_T,
\end{split}
\end{equation}
where we used $\sum_{T\in{\cal T}_h}\langle a\nabla w\cdot \bn, e_b\rangle_{\partial T}=\langle a\nabla w\cdot \bn, e_b\rangle_{\partial \Omega}=0$ due to the facts that $a\nabla w\cdot \bn=0$ on $\Gamma_N$ and $e_b=0$ on $\Gamma_D$.

It follows from (\ref{disgradient*}) and (\ref{l}) that
\begin{equation*}
\begin{split}
(a\nabla_w e_h, \nabla_w (Q_h w))_T =&(a\nabla_w e_h, {\cal Q}^{(k-1)}_h  \nabla w)_T\\
 = & (a\nabla e_0, {\cal Q}_h^{(k-1)} \nabla w)_T-\langle e_0-e_b, a{\cal Q}_h^{(k-1)}\nabla w\cdot \bn \rangle_{\partial T}\\
  = & (a\nabla e_0, \nabla w)_T-\langle e_0-e_b, a {\cal Q}_h^{(k-1)}  \nabla w\cdot \bn \rangle_{\partial T},
\end{split}
\end{equation*}
which gives {\color{black}
\begin{equation} \label{naterm}
\begin{split}
 (a\nabla e_0, \nabla w)_T&=(a\nabla_w e_h, {\cal Q}^{(k-1)}_h  \nabla w)_T +\langle e_0-e_b, a{\cal Q}_h^{(k-1)}  \nabla w\cdot \bn \rangle_{\partial T}\\
   &=(a\nabla_w e_h,  \nabla w)_T +\langle e_0-e_b, a{\cal Q}_h^{(k-1)}  \nabla w\cdot \bn \rangle_{\partial T}.
\end{split}
\end{equation}
}
Substituting (\ref{naterm}) into (\ref{2.14:800:10:L2}) leads to {\color{black}
\begin{equation}\label{l2norm}
\begin{split}
 \|e_0\|^2 
=&  \sum_{T\in{\cal T}_h}(a\nabla_w e_h,\nabla  w)_T +(\bb e_0,  \nabla w)_T +\langle e_0-e_b, a({\cal Q}_h^{(k-1)}-I)  \nabla w\cdot \bn \rangle_{\partial T} \\
=& b(e_h,w) +\sum_{T\in{\cal T}_h} \langle e_n, w\rangle_{\pT}+\langle e_0-e_b, a({\cal Q}_h^{(k-1)}-I)  \nabla w\cdot \bn \rangle_{\partial T} \\ 
=&  c(\epsilon_h, {\cal Q}_h^{(k)} w)
+\ell_u({\cal Q}_h^{(k)} w) +b(e_h, (I-{\cal Q}_h^{(k)})w)\\
&\ +\sum_{T\in \T_h}\langle e_0-e_b, a({\cal Q}_h^{(k-1)}-I)  \nabla w\cdot \bn \rangle_{\partial T} \\
=& I_1+I_2+I_3+I_4,
\end{split}
\end{equation}
where in the second step, we have used the fact that  $\sum_{T\in{\cal T}_h} \langle e_n, w\rangle_{\pT}=\langle e_n, w\rangle_{\partial \Omega}=0$ due to the facts that $w=0$ on $\Gamma_D$ and $e_n=0$ on $\Gamma_N$, and $I_i (i=1, \cdots, 4)$ is defined accordingly.
}

We shall estimate each of the four terms $I_i$ for $i=1, \cdots, 4$ in (\ref{l2norm}). As to the term $I_1$, for the case of $l=k$ where $\tau_1=0$ and $\tau_2=0$,  we have
\begin{equation}\label{l2norm1-1-1}
I_1=c(\epsilon_h, {\cal Q}_h^{(k)}w)=0.
\end{equation}
{\color{black} 
For the case of $l=k-1$,  we have, from the Cauchy-Schwarz inequality and (\ref{erres}), 
\begin{equation}\label{l2norm1-1}
\begin{split}
I_1=& \tau_1   \sum_{T\in \T_h}h_T^2 (\nabla \epsilon_h, \nabla {\cal Q}_h^{(k)}w)_T+ \tau_2   \sum_{i, j=1}^d \sum_{T\in \T_h}h_T^4(\partial_{ij}^2\epsilon_h, \partial_{ij}^2 ({\cal Q}_h^{(k)}w))_T\\
\le & \3bar \epsilon_h \3bar_{M_h} \left(  \sum_{T\in \T_h}\Big(\tau_1h_T^2 \|\nabla {\cal Q}_h^{(k)}w\|_T^2+ \tau_2   \sum_{i, j=1}^d h_T^4\|\partial_{ij}^2 {\cal Q}_h^{(k)}w\|_T^2\Big)\right)^{\frac 12}\\
 \leq & 
 \left\{\begin{split}
Ch\tau_1^{\frac 12}\3bar \epsilon_h \3bar_{M_h}\|w\|_{1},\quad & k=1, l=k-1,
 \\
Ch^{1+s}\tau_2^{\frac 12}\3bar \epsilon_h \3bar_{M_h} \|w\|_{1+s},\quad &  k\geq 2, l=k-1.
  \end{split}\right.
\end{split}
\end{equation}
}
Here in the last step, we have used the fact that $\tau_2=0$ for $k=1$ and $\tau_1=0$ for $k\ge 2$, and the inverse inequality
\[
   | {\cal Q}_h^{(k)}w|_{2}\le C h^{s-1} | {\cal Q}_h^{(k)}w|_{1+s}\le C h^{s-1} \|w\|_{1+s},\ \ \frac 12<s\le 1. 
\]

As to the term $I_2$, for the case of $l=k$, we have from (\ref{lu}) that
$$
I_2=\ell_u({\cal Q}_h^{(k)} w)=0.
$$
{\color{black} For the case of $l=k-1$,  by the same argument as what we did in \eqref{aij}, we have 
\begin{equation}\label{l2norm1} 
\begin{split}
&|I_2|=|\ell_u({\cal Q}_h^{(k)} w)|\\
\leq &Ch^{k-1}\|u\|_{k+1}\Big( \sum_{T\in {\cal T}_h} \|(I-I_h^l){\cal Q}_h^{(k)} w)\|_{T}^2+h^2_T\| \nabla((I-I_h^l){\cal Q}_h^{(k)} w)\|_{T}^2\Big)^{\frac{1}{2}}\\
 \leq & 
 \left\{\begin{split}
Ch^{k}\|u\|_{k+1}\|w\|_{1},\quad & k=1, l=k-1,
 \\
Ch^{k+s}\|u\|_{k+1} \|w\|_{1+s},\quad &  k\geq 2, l=k-1.
  \end{split}\right.
\end{split}
\end{equation}
Here for any function $v$, $I_h^lv$ denotes the cell average and linear interpolation of $v$
on each element $T\in {\cal T}_h$ for  $l=0$ and $l\ge 1$, resepctively. 

To estimate $I_3$, we note that 
\begin{eqnarray*}
 I_3&=& 
\sum_{T\in{\cal T}_h}\langle (a\nabla_w e_h+\bb e_0)\cdot\bn-e_n, (I-{\cal Q}_h^{(k)})w\rangle_{\pT}\\
&& -\sum_{T\in{\cal T}_h}(\nabla\cdot(a\nabla_w e_h+\bb e_0),  (I-{\cal Q}_h^{(k)}) w)_T  \\
 &=& I_{31}-I_{32}. 
\end{eqnarray*}
To estimate $I_{31}$,  we have from the Cauchy-Schwarz inequality, the trace inequality (\ref{x}),  (\ref{ehe0}), the estimate (\ref{3.3-2}) with $m=s$ that (with $F_e=a\nabla e_0+\bb e_0$)
\begin{eqnarray*}
&&|I_{31}|\\
&= &\sum_{T\in{\cal T}_h}\langle F_e\cdot\bn-e_n, (I-{\cal Q}_h^{(k)})w\rangle_{\pT} + \langle (a\nabla_w e_h-a\nabla e_0)\cdot\bn, (I-{\cal Q}_h^{(k)})w\rangle_{\pT}\\
&\leq & \Big\{\Big(\sum_{T\in {\cal T}_h}\|F_e\cdot\bn-e_n\|_{\pT}^2\Big)^{\frac{1}{2}} + \Big(\sum_{T\in {\cal T}_h}\| (a\nabla_w e_h-a\nabla e_0)\cdot\bn\|_{\pT}^2\Big)^{\frac{1}{2}}\Big\}\\
&& \ \cdot
\Big(\sum_{T\in {\cal T}_h}\|(I-{\cal Q}_h^{(k)})w\|_{\pT}^2\Big)^{\frac{1}{2}}\\
  &\leq &C\Big\{ \Big(\sum_{T\in {\cal T}_h}  \|F_e\cdot\bn-e_n\|_{\pT}^2     \Big)^{\frac{1}{2}} + \Big(\sum_{T\in {\cal T}_h} h_T^{-1}\|(a\nabla_w e_h-a\nabla e_0)\cdot\bn\|_{T}^2 \Big)^{\frac{1}{2}} \Big\}h^{s+\frac{1}{2}}\|w\|_{1+s}\\
 & \leq &C\Big\{h^{\frac{1}{2}}\3bar e_h\3bar_{W_h} + \Big(\sum_{T\in {\cal T}_h} h_T^{-2}\|e_0- e_b \|_{\pT}^2 \Big)^{\frac{1}{2}} \Big\}h^{s+\frac{1}{2}}\|w\|_{1+s}\\
&\leq& C h^{s}\3bar e_h\3bar_{W_h}\|w\|_{1+s}.
\end{eqnarray*}

Similarly, for the term $I_{32}$, we have from the Cauchy-Schwarz inequality, the estimate (\ref{3.3-2}) with $m=s$,   (\ref{ehe0}),  the inverse inequality that
\begin{eqnarray*} 
| I_{32}| 
&=&  \left|\sum_{T\in{\cal T}_h} (\nabla\cdot F_e,  (I-{\cal Q}_h^{(k)}) w)_T+(\nabla\cdot(a\nabla_w e_h-a\nabla e_0),  (I-{\cal Q}_h^{(k)}) w)_T\right|\\ 
&\leq &C\Big\{ \Big(\sum_{T\in {\cal T}_h}  \|\nabla\cdot F_e\|_{T}^2  \Big)^{\frac{1}{2}}   +\Big(\sum_{T\in {\cal T}_h} h_T^{-2}\|a\nabla_w e_h-a\nabla e_0\|_{T}^2  \Big)^{\frac{1}{2}}\Big\}h^{1+s}\|w\|_{1+s}\\
&\leq &C\Big\{ \Big(\sum_{T\in {\cal T}_h}  \|\nabla\cdot F_e\|_{T}^2  \Big)^{\frac{1}{2}}   +\Big(\sum_{T\in {\cal T}_h} h_T^{-3}\|e_0- e_b\|_{\pT}^2  \Big)^{\frac{1}{2}}\Big\}h^{1+s}\|w\|_{1+s}\\
&\leq &C\Big\{ \Big(\sum_{T\in {\cal T}_h}  \|\nabla\cdot(a\nabla e_0+\bb e_0)\|_{T}^2  \Big)^{\frac{1}{2}}   + \3bar e_h \3bar_{W_h} \Big\}h^{1+s}\|w\|_{1+s}.
\end{eqnarray*}
  Consequently, 
 \[
    |I_3|\le C \Big\{ \Big(\sum_{T\in {\cal T}_h}  \|\nabla\cdot(a\nabla e_0+\bb e_0)\|_{T}^2  \Big)^{\frac{1}{2}}   + \3bar e_h \3bar_{W_h} \Big\}h^{1+s}\|w\|_{1+s}.
 \]

}

As to the term $I_4$, we have from the Cauchy-Schwarz inequality, trace inequality (\ref{trace-inequality}), the estimate  (\ref{3.3}) with $m=s$,  that
\begin{equation}\label{l2norm2}
\begin{split}
|I_{4}| 
\leq & \Big(\sum_{T\in {\cal T}_h}  \|e_0-e_b\|_{\partial T}^2 \Big)^{\frac{1}{2}}\Big(\sum_{T\in {\cal T}_h}  \|a({\cal Q}_h^{(k-1)}-I)  \nabla w\cdot \bn \|_{\partial T}^2 \Big)^{\frac{1}{2}}\\
\leq & \Big(\sum_{T\in {\cal T}_h} h_T^{-3} \|e_0-e_b\|_{\partial T}^2 \Big)^{\frac{1}{2}}\Big(\sum_{T\in {\cal T}_h}h_T^3  \|a({\cal Q}_h^{(k-1)}-I)  \nabla w\cdot \bn \|_{\partial T}^2 \Big)^{\frac{1}{2}}\\
\leq &C \3bar e_h\3bar_{W_h} \Big(\sum_{T\in {\cal T}_h} h_T^2\|({\cal Q}_h^{(k-1)}-I)  \nabla w\|_{T}^2  +h_T^4 \|({\cal Q}_h^{(k-1)}-I)  \nabla w\|^2_{1, T} \Big)^{\frac{1}{2}}\\
\leq & C\3bar e_h\3bar_{W_h}  h^{s+1}\|w\|_{1+s}. 
\end{split}
\end{equation}

Substituting  (\ref{l2norm1-1-1})- (\ref{l2norm2}) into (\ref{l2norm}) and using the  regularity assumption (\ref{regul}) with the error estimates (\ref{erres}) and (\ref{rese0}) gives (\ref{e0}). This completes the proof of this theorem.
\end{proof}

We shall establish the error estimates for the two boundary components
$u_b$ and $u_n$ of the PDWG solution $u_h=\{u_0, u_b, u_n\}$ in the usual $L^2$ norms defined as follows:
\begin{eqnarray}\label{EQ:eb-eg-L2norm}
\|e_b\| &:=&\|u_b- Q_b^{(k)} u\|=\Big(\sum_{T\in {\cal T}_h} h_T\|e_b\|_{\partial
T}^2\Big)^{\frac{1}{2}},\\
 \|e_n\| &:=&\|u_n - Q_n^{(l)} ((a\nabla u+\bb u)\cdot\bn)\|=\Big(\sum_{T\in {\cal
T}_h} h_T\|e_n\|_{\partial T}^2\Big)^{\frac{1}{2}}.
\end{eqnarray}

\begin{theorem}\label{Thm:L2errorestimate-ub}
Under the assumptions of Theorem \ref{Thm:L2errorestimate}, there
exists a constant $C$ such that
\begin{eqnarray}\label{eb}
\|e_b\| &\leq   \left\{\begin{split}
C h^{k+s}\|u\|_{k+1},\quad &l=k,\\
C(1+\tau_1^{\frac{1}{2}}(1+\tau_1^{-\frac{1}{2}}))h^{k}\|u\|_{k+1},\quad & k=1, l=k-1,   \\
C(1+\tau_2^{\frac{1}{2}})(1+\tau_2^{-\frac{1}{2}})h^{k+s}\|u\|_{k+1},\quad &k\geq 2,  l=k-1.
  \end{split}\right.   \\
\|e_n\| &   \leq  \left\{\begin{split}
C h^{k+s-1}\|u\|_{k+1},\quad &l=k,\\
C(1+\tau_1^{\frac{1}{2}}(1+\tau_1^{-\frac{1}{2}}))h^{k-1}\|u\|_{k+1},\quad & k=1, l=k-1,   \\
C(1+\tau_2^{\frac{1}{2}})(1+\tau_2^{-\frac{1}{2}})h^{k+s-1}\|u\|_{k+1},\quad &k\geq 2,  l=k-1.
  \end{split}\right.    \label{eg}
\end{eqnarray}
provided that the meshsize $h$ is sufficiently small.
\end{theorem}

\begin{proof} On each element $T\in\T_h$, we have from the triangle
inequality that
$$
\|e_b\|_{\pT} \leq \|e_0\|_{\pT} + \|e_b-e_0\|_\pT.
$$
Thus, by (\ref{EQ:eb-eg-L2norm}), and the trace inequality (\ref{x}), we obtain
\begin{equation*}
\begin{split}
\|e_b\|^2 = \sum_{T\in\T_h} h_T \|e_b\|_{\pT}^2 & \leq C \sum_{T\in\T_h}
h_T\|e_0\|_{\pT}^2 + C h^4 \sum_{T\in\T_h}
h_T^{-3}\|e_b-e_0\|_\pT^2\\
& \leq C (\|e_0\|^2_0 + h^4 \3bar e_h\3bar_{W_h}^2),
\end{split}
\end{equation*}
which, together with the error estimates (\ref{erres}) and
(\ref{e0}), gives rise to (\ref{eb}).

To derive (\ref{eg}),   applying the same approach to the error
component $e_n$  by using triangle inequality, trace inequality  (\ref{x}) and inverse inequality gives
\begin{equation*}
\begin{split}
	\|e_n\| ^2
 = & \sum_{T\in\T_h} h_T \|e_n\|_{\pT}^2                                                                                                                 \\
	\leq          & \sum_{T\in\T_h} h_T \|(a\nabla e_0+\bb e_0)\cdot \bn-e_n\|_{\pT}^2 +h_T \|(a\nabla e_0+\bb e_0)\cdot \bn \|_{\pT}^2                                 \\
	\leq          & Ch^2\sum_{T\in\T_h} h_T^{-1} \|(a\nabla e_0+\bb e_0)\cdot \bn-e_n\|_{\pT}^2 +  \sum_{T\in\T_h}(h_T \|a\nabla e_0\|^2_{\pT}+h_T\|\bb e_0 \|_{\pT}^2) \\
	\leq          & C (h^2 \3bar e_h\3bar_{W_h}^2+  \sum_{T\in\T_h}
(h_T h_T^{-3} \|e_0\|^2_{T}+h_T h_T^{-1}\|e_0 \|_{T}^2))                                            \\
	\leq          & C(h^2 \3bar e_h\3bar_{W_h}^2+   
h^{-2} \|e_0\|^2),
\end{split}
\end{equation*}
which, together with the error estimates (\ref{erres}) and
(\ref{e0}), gives rise to the error estimate (\ref{eg}).
\end{proof}

\section{Numerical Results}
Two types of domains are considered in the numerical experiment: (1) an unit square domain $\Omega_1=[0,1]^2$, and (2) a L-shaped domain $\Omega_2$ with vertices $(0, 0), (0.5, 0), (0.5, 0.5), (1, 0.5), (1,1), (0, 1)$. In all the computation, the finite element partition $\T_h$ is obtained through a successive uniform refinement of a coarse triangulation of the domain $\Omega$ by dividing each coarse element into four congruent sub-elements by connecting the mid-points of the three edges of the triangle. The right-hand side function $f$, the Dirichlet boundary data $g_1$ and the Neumann boundary data $g_2$ are set correspondingly. For simplicity, the parameters in the PDWG numerical scheme \eqref{32}-\eqref{2} are chosen as $\tau_1=\tau_2=1$. 
    
The finite element spaces for the primal variable $u_h$ and the dual variable $\lambda_h$ are given by
\begin{equation*}
W_{k, h}=\{u_h=\{u_0, u_b, u_n\}: \ u_0\in P_k(T), u_b\in P_k(e), u_n\in P_{l}(e),   \forall e\subset\pT, \forall T\in {\cal T}_h\},
\end{equation*}
\begin{equation*}
M_{k, h}=\{\lambda_h: \ \lambda_h|_T \in P_k(T),\ \forall T\in {\cal T}_h\},
\end{equation*}
where  $l=k$ or $l=k-1$. The primal-dual weak Galerkin scheme (\ref{32})-(\ref{2}) is implemented for the case of $k=1$ and $k=2$.
  
Denote by $e_h=\{e_0, e_b, e_n\}=u_h-Q_h u$ the error function. The following $L^2$ norms are used to measure the errors:
$$
  \|e_0\|  = \Big(\sum_{T\in {\cal T}_h} \int_T e_0^2 dT\Big)^{\frac{1}{2}},
 \qquad 
  \|\nabla e_0\|  = \Big(\sum_{T\in {\cal
  		T}_h} \int_T ( \nabla e_0 ) ^2 dT\Big)^{\frac{1}{2}},
$$
$$
   \|e_b\|  = \Big(\sum_{T\in {\cal T}_h}h_T \int_{\partial T}  e_b^2 ds\Big)^{\frac{1}{2}},\qquad   
  \|e_n\|  = \Big(\sum_{T\in {\cal T}_h}h_T \int_{\partial T}  e_n^2 ds\Big)^{\frac{1}{2}}.
$$

{\bf Test Example 1 (Constant diffusion $a$ and convection $\bb$).}
 The diffusion tensor $a\in \mathbb R^{2\times 2}$  and the convection tensor ${\bf b}\in \mathbb R^2$ are taken by constants as follows:
\[
  a_{11}=1,\ \ a_{12}=a_{21}=1,\ \ a_{22}=6;\ \ b_1=1, \ \ b_2=1.
\]
The exact solution is given by $u(x, y)=\sin(\pi x)\sin(\pi y)$. The domain the unit square domain $\Omega_1$. The Neumann boundary is $\Gamma_N=\{(0, y): y\in [0,1]\}$, and the rest of the boundary is of Dirichlet. 
   
 Tables \ref{1}-\ref{122} demonstrate the approximation errors and the corresponding convergence rates for $k=1$ and $k=2$ with $l=k$ and $l=k-1$, respectively.
For the case of $l=k$, we observe from Table  \ref{1} that the convergence orders for $e_0$ and $e_b$ in the discrete $L^2$ norm are both of an optimal order ${\cal O}(h^{k+1})$, and the convergence order for $e_n$ in the discrete $L^2$ norm is of an optimal order ${\cal O}(h^{k})$, for  $k=1$ and $k=2$ respectively, which are all consistent with the theoretical results in Theorems \ref{Thm:L2errorestimate} -  \ref{Thm:L2errorestimate-ub}. For the case of $l=k-1$, we can see from Table \ref{122} that  
  the convergence rates for $e_0$ and $e_b$ in the discrete $L^2$ norm are of an order ${\cal O}(h^{k+1})$, and the convergence rate for $e_n$ in the discrete $L^2$ norm is of an order ${\cal O}(h^{k})$ for $k=1$ and $k=2$, respectively. Note that for the case of $l=k-1$ and $k=2$, the convergence rates for $\|e_0\|, \|e_b\|$ and $\|e_n\|$ are consistent with the theoretical results developed in  Theorems \ref{Thm:L2errorestimate}
- \ref{Thm:L2errorestimate-ub}; while for the case of $l=k-1$ and $k=1$, the convergence rates for $\|e_0\|, \|e_b\|$ and $\|e_n\|$ are of an order which is 1 order higher than the expected convergence order given by \eqref{e0} and \eqref{eb}-\eqref{eg}, respectively.

\begin{table}[htbp]\caption{Various errors and corresponding convergence rates  for $k=1,2$ with $l=k$ on  $\Omega_1$.}\label{1}\centering
\begin{threeparttable}
        \begin{tabular}{c |c |c | c | c | c | c | c | c | c  |}
        \hline
    & $1/h$ &   $\|e_b\|$ & rate& $\|e_n\|$  & rate & $\|\nabla e_0\|$ & rate & $\|e_0\|$ & rate \\
       \hline \cline{2-10}
  &4& 3.48e-1 &         --   &7.88e-0&         --   &7.53e-1&       --  & 7.97e-2 &      -- \\
   &8&  8.72e-2  & 2.00  & 3.75e-0 &  1.07  & 3.38e-1 &  1.16   &2.10e-2  & 1.92\\
$k=1$   &16&2.01e-2   &2.11  & 1.72e-0  & 1.12   &1.56e-1  & 1.11  & 4.92e-3  & 2.09\\
  &32& 4.77e-3  & 2.08 &  8.31e-1   &1.05  & 7.58e-2   &1.04 &  1.16e-3  & 2.08\\
  &64& 1.17e-3   &2.02&   4.11e-1  & 1.02 &  3.74e-2  & 1.02&   2.85e-4  & 2.02\\
\hline
& $1/h$ &   $\|e_b\|$ & rate& $\|e_n\|$  & rate & $\|\nabla e_0\|$ & rate & $\|e_0\|$ & rate \\
       \hline \cline{2-10}
 &   2& 1.59e-1  &   --  & 6.89e-0 &        --  & 7.90e-1 &  -- &  9.94e-2  &       -- \\
     &   4& 2.87e-2  & 2.47  & 1.43e-0  & 2.27  & 2.28e-1  &  1.79 &  1.63e-2 &  2.61 \\
 $k=2$ & 8&  3.63e-3 &  2.98  & 3.48e-1  & 2.03 &  6.28e-2  & 1.86 &  2.28e-3  & 2.84 \\
   &     16& 4.60e-4  & 2.98 &  8.79e-2 &  1.99 &  1.65e-2 &  1.93  & 3.01e-4  & 2.92  \\
  &    32& 5.89e-5 &  2.96  &  2.21e-2  & 1.99 &  4.22e-3 &  1.97 &  3.86e-5 &  2.96  \\
\hline
 \end{tabular}
 \end{threeparttable}
\end{table}

\begin{table}[htbp]\caption{Various errors and corresponding convergence rates  for $k=1,2$ with $l=k-1$ on  $\Omega_1$.}\label{122}\centering
\begin{threeparttable}
        \begin{tabular}{c |c |c | c | c | c | c | c | c | c  |}
        \hline
 & $1/h$ &   $\|e_b\|$ & rate& $\|e_n\|$  & rate & $\|\nabla e_0\|$ & rate & $\|e_0\|$ & rate \\
       \hline \cline{2-10}
  &4&3.77e-1  &     -- & 7.57e-0 &     -- & 7.19e-1&      --   &9.30e-2&      --\\
  &8& 9.77e-2  & 1.95  & 3.59e-0 &  1.07  & 3.28e-1&   1.13  &2.49e-2&   1.90\\
 $k=1$ &16& 2.46e-2   &1.99 & 1.73e-0  & 1.05   &1.56e-1 &  1.07  & 6.27e-3 &  2.00\\
  &32& 6.16e-3   &2.00  & 8.54e-1   &1.02  & 7.65e-2 &  1.03 &  1.57e-3  & 2.00\\
  &64& 1.54e-3  & 2.00 &  4.24e-1  & 1.01 &  3.79e-2 &  1.01&   3.91e-4   &2.00\\
\hline
 & $1/h$ &   $\|e_b\|$ & rate& $\|e_n\|$  & rate & $\|\nabla e_0\|$ & rate & $\|e_0\|$ & rate \\
       \hline \cline{2-10}
  &2& 2.16e-1 &        --  & 6.00e-0&      -- & 8.56e-01 &   --   &1.12e-1&      -- \\
  &4& 3.39e-2  & 2.67   &1.42e-0 &  2.08   &2.44e-01  & 1.81 & 1.75e-2 &  2.67\\
$k=2$  &8& 4.08e-3   &3.05  & 3.54e-1  & 2.00  & 6.54e-2   &1.90   &2.39e-3 &  2.87\\
  &16& 4.88e-4  & 3.06 &  8.90e-2   &1.99 &  1.69e-2  & 1.96  & 3.08e-4  & 2.95\\
  &32& 6.05e-5   &3.01&   2.23e-2  & 2.00&   4.26e-3   &1.98 &  3.90e-5   &2.98\\
 \hline
 \end{tabular}
 \end{threeparttable}
\end{table}

{\bf Test Example 2 (Continuous diffusion $a$ and convection $\bb$).} We choose the diffusion tensor $a \in \mathbb{R}^{2\times 2}$ and the convection tensor ${\bf b} \in \mathbb{R}^{2}$ in the model problem \eqref{model}  as continuous functions as follows: 
  \[
     a_{11}=1+x,\ \ a_{12}=a_{21}=0,\ \  a_{22}=1+y;\ \   b_1=e^{1-x}, \ \ b_2=e^{xy}.
  \]
 The Neumann boundary is $\Gamma_N=\{(0,y): y\in [0,1]\}$ and the Dirichlet boundary is $\Gamma_D=\partial\Omega\setminus
     \Gamma_N$.  The exact solution is given by $u(x,y)=\sin(x)\cos( y)$. The domain is the unit square $\Omega_1$.

Tables \ref{3}-\ref{4} demonstrate the numerical errors and the convergence rates arising from the PDWG scheme (\ref{32})-(\ref{2}) for the convection-diffusion model problem \eqref{model}. We observe from Tables \ref{3}-\ref{4} that the numerical performance is the same as those in Tables \ref{1}-\ref{122}.  

\begin{table}[htbp]\caption{Various errors and corresponding convergence rates  for $k=1,2$ with $l=k$ on $\Omega_1$.}\label{3}\centering
\begin{threeparttable}
        \begin{tabular}{c |c |c | c | c | c | c | c | c | c  |}
        \hline
    & $1/h$ &   $\|e_b\|$ & rate& $\|e_n\|$  & rate & $\|\nabla e_0\|$ & rate & $\|e_0\|$ & rate \\
       \hline \cline{2-10}
    &4&7.65e-3   &     --  & 2.0277e-01    &   --  & 4.60e-2   &   --  & 2.34e-3    &  --\\
       &8& 1.96e-3 &  1.96 & 9.7447e-02  & 1.06&   2.22e-2 &  1.05 &  5.92e-4  & 1.98 \\
 $k=1$ &16& 4.96e-4 &  1.99 &  4.7834e-02  & 1.03 &  1.10e-2 &  1.02 &  1.48e-4 &  2.00 \\
      &32& 1.24e-4 &  2.00 &  2.3696e-02  & 1.01 &  5.46e-3 &  1.01 &  3.69e-5 &  2.00\\
     &64& 3.10e-5 &  2.00 &  1.1792e-02 &  1.01 &  2.72e-3  & 1.00&  9.21e-6 &  2.00\\
 \hline
   & $1/h$ &   $\|e_b\|$ & rate& $\|e_n\|$  & rate & $\|\nabla e_0\|$ & rate & $\|e_0\|$ & rate \\
       \hline \cline{2-10}
    &2&  1.06e-2  &   --  & 1.27e-1  &   --  & 3.22e-2 &     --  & 4.33e-3  &   -- \\
      &4&  9.07e-4  & 3.54  & 2.57e-2 &  2.31 &  7.19e-3 &  2.16  & 5.03e-4  & 3.10\\
$k=2$ &8&  9.04e-5  & 3.33  & 6.07e-3 &  2.08 & 1.75e-3 &  2.04  & 6.16e-5  & 3.03\\
      &16&  9.86e-6  & 3.20  & 1.48e-3  & 2.03 &  4.33e-4 &  2.01  & 7.64e-6  & 3.01\\
      &32&  1.14e-6  & 3.11  & 3.67e-4 &  2.01 &  1.08e-4 &  2.00  & 9.53e-7 &  3.00\\
   \hline
 \end{tabular}
 \end{threeparttable}
\end{table}

\begin{table}[htbp]\caption{Various errors and corresponding convergence rates  for $k=1,2$ with $l=k-1$ on  $\Omega_1$.}\label{4}\centering
\begin{threeparttable}
        \begin{tabular}{c |c |c | c | c | c | c | c | c | c  |}
        \hline
 & $1/h$ &   $\|e_b\|$ & rate& $\|e_n\|$  & rate & $\|\nabla e_0\|$ & rate & $\|e_0\|$ & rate \\
       \hline \cline{2-10}
  &4& 2.79e-2 &   --  & 4.69e-1  &   --  & 8.16e-2 &   --  & 9.72e-3 &   -- \\
       &8& 6.53e-3 &  2.10  & 2.20e-1  & 1.09  & 2.87e-2 &  1.51  & 1.81e-3 &  2.42\\
$k=1$  &16& 1.60e-3 &  2.03 &  1.07e-1  & 1.04  & 1.20e-2 &  1.26  & 3.98e-4 &  2.18 \\
       &32& 3.97e-4 &  2.01  & 5.32e-2  & 1.01  & 5.62e-3 &  1.09  & 9.60e-5 &  2.05 \\
       &64& 9.90e-5 &  2.00  & 2.65e-2  & 1.00  & 2.75e-3  & 1.03  & 2.38e-5 &  2.01 \\
 \hline
   & $1/h$ &   $\|e_b\|$ & rate& $\|e_n\|$  & rate & $\|\nabla e_0\|$ & rate & $\|e_0\|$ & rate \\
       \hline \cline{2-10}
     &2& 1.14e-2 &  -- &  1.69e-1 &  -- &  3.55e-2 &  -- & 4.96e-3  & --\\
       &4& 1.10e-3 &  3.37 &  3.89e-2 &  2.11 &  7.54e-3 &  2.24  & 5.49e-4 &  3.17\\
 $k=2$ &8& 1.18e-4 &  3.22 &  9.49e-3 &  2.04 &  1.79e-3 &  2.08  & 6.59e-5  & 3.06\\
       &16& 1.37e-5 &  3.11 &  2.35e-3 &  2.01 &  4.39e-4 &  2.02  & 8.14e-6  & 3.02\\
       &32& 1.65e-06 &  3.05&   5.85e-4 &  2.01&   1.09e-4&   2.01 &  1.01e-6 &  3.01\\
\hline
 \end{tabular}
 \end{threeparttable}
\end{table}

{\bf Test Example 3 (Convection-dominated diffusion problem).} We consider the diffusion tensor $a \in \mathbb{R}^{2\times 2}$ and the convection tensor ${\bf b} \in \mathbb{R}^{2}$ given by
\[
     a_{11}=a_{12}=\epsilon,\ \ a_{12}=a_{21}=0, \ \  b_1=1, \ \ b_2=1,
\]
where $\epsilon$ assumes some small and positive constants. The exact solution is given by $u(x, y)=(x+0.5)(y+0.5)e^{1-x}e^{y}$ with the full Dirichlet boundary condition $\Gamma_D=\partial\Omega_1$.

Tables \ref{6}-\ref{7} show the approximation errors and the convergence rates for the convection-dominated diffusion problem with $\epsilon=10^{-10}$ on the unit square domain $\Omega_1$. For the case of $l=k$ and $k=1$, we observe from Table \ref{6}  that  the convergence rates for the errors $e_0$, $e_b$ and $e_n$ in the discrete $L^2$-norm are of an order ${\cal O}(h^2)$, respectively, which are consistent with the theory for both $\|e_0\|$ and $\|e_b\|$; and are of one order higher than the expected order ${\cal O}(h)$ for $\|e_n\|$. For the case of $l=k-1$ with $k=1$, the convergence rates of  $e_0$, $e_b$, and $e_n$ in the discrete $L^2$-norm shown in Table \ref{7}  are all of order ${\cal O}(h)$, which are consistent with the expected order for both $\|e_0\|$ and $\|e_b\|$; and is of one order higher than the expected convergence rate given by \eqref{eg} for $\|e_n\|$.  
For the case of $l=k$ and $k=2$, we observe from Table \ref{6} that  the convergence rates for the errors $e_0$, $e_b$ and $e_n$ in the discrete $L^2$-norm are of order ${\cal O}(h^2)$, which are consistent with the theory for $\|e_n\|$; and are of one order lower than the expected optimal order ${\cal O}(h^{3})$ for both $\|e_0\|$ and $\|e_b\|$.  For the case of $l=k-1$ and $k=2$, we see from Table \ref{7} that  the convergence rates for the errors $e_0$, $e_b$ and $e_n$ in the discrete $L^2$-norm are of  order ${\cal O}(h^2)$, which are consistent with the theory for $\|e_0\|$ and $\|e_b\|$; and are of one order higher than the expected optimal order ${\cal O}(h)$ for  $\|e_n\|$. 
 
Table \ref{8-1} shows the approximation errors and the convergence rates for $\epsilon=10^{-2}$ on the unit square domain $\Omega_1$ when $k=2$ is employed. 
For the case of $l=k$, we observe from Table \ref{8-1} that the convergence rates for the errors $e_0$ and $e_b$ in the discrete $L^2$-norm are of optimal order ${\cal O}(h^3)$, which is consistent with the theory; and  the convergence rate for the error  $e_n$ in the discrete $L^2$-norm is one order higher than the expected optimal order ${\cal O}(h^2)$, which outperforms the theory. 
For the case of $l=k-1$, the convergence rates of  $e_0$, $e_b$ and $e_n$ in the discrete $L^2$-norm are of one order higher than the expected optimal order, which are better than the theory. It can be seen that the convergence
rates are improved and the theoretical results in Theorems \ref{Thm:L2errorestimate}-\ref{Thm:L2errorestimate-ub} are recovered for the case of $\epsilon=10^{-2}$. The results indicate that the diffusion coefficient  $a$ has an influence on the convergence rate for $k=2$.
 
Tables \ref{6-1}-\ref{7-1} show the  approximation errors and  corresponding convergence rates for $k=1, 2$ with $l=k$ and $l=k-1$, respectively, on the L-shaped domain $\Omega_2$. In both the case of $l=k$ for $k=1, 2$ and $l=k-1$ for $k=2$, we observe a $(k+1)$-th order of convergence for  $\|e_0\|$ and $\|e_b\|$, and a $k$-th order of convergence for $\|e_n\|$, which are consistent with our  theoretical results established in Theorems \ref{Thm:L2errorestimate}-\ref{Thm:L2errorestimate-ub}. As for the case of $l=k-1$ and $k=1$,  we observe from Table \ref{7-1} that, the convergence rates of  $\|e_0\|$, $\|e_b\|$, $\|e_n\|$ are all of order ${\cal O}(h)$. Note that the convergence results for $\|e_0\|$ and $\|e_b\|$ for the case of $l=k-1$ and $k=1$ are consistent with the theoretical findings in Theorems \ref{Thm:L2errorestimate}-\ref{Thm:L2errorestimate-ub}; while for $\|e_n\|$, the convergence rate is $1$ order higher than the error estimate given by \eqref{eg}.

  \begin{table}[htbp]\caption{Various errors and corresponding convergence rates  for $k=1,2$ with $l=k$ and $\epsilon=10^{-10}$ on  $\Omega_1$.}\label{6}\centering
\begin{threeparttable}
        \begin{tabular}{c |c |c | c | c | c | c | c | c | c  |}
        \hline
 & $1/h$ &   $\|e_b\|$ & rate& $\|e_n\|$  & rate & $\|\nabla e_0\|$ & rate & $\|e_0\|$ & rate \\
       \hline \cline{2-10}
      &4&  9.52e-2&    -- &  1.11e-1 &     --&  7.76e-01 &      --  & 2.78e-2 &     --   \\
       &8& 2.64e-2  & 1.85  & 2.99e-2  & 1.90 &  4.07e-01 &  0.93  & 7.53e-3  & 1.88\\
$k=1$  &16& 6.77e-3 &  1.96  & 7.65e-3  & 1.96  & 2.06e-01 &  .098  & 1.92e-3  & 1.97\\
       &32& 1.70e-3 &  1.99  & 1.93e-3  & 1.99 & 1.03e-01 &  1.00  & 4.83e-4  & 1.99\\
       &64&  4.27e-4 &  2.00  & 4.84e-4  & 2.00  & 5.18e-02 &  1.00  & 1.21e-4  & 2.00 \\
\hline
& $1/h$ &   $\|e_b\|$ & rate& $\|e_n\|$  & rate & $\|\nabla e_0\|$ & rate & $\|e_0\|$ & rate \\
       \hline \cline{2-10}
     &2& 3.73e-2 &    -- &  5.77e-2 &      -- &  2.27e-1 &     -- &  1.28e-2 &  --\\
      &4& 6.65e-3 &  2.49&  9.50e-3 &  2.60 &  6.44e-2 &  1.82 &  2.17e-3  & 2.55\\
$k=2$ &8&  1.36e-3  & 2.29&   1.86e-3  & 2.35&  1.97e-2 &  1.71 &  4.64e-4 &  2.23\\
    &16&  3.13e-4 &  2.11 &  4.23e-4 &  2.14 &  7.38e-3  & 1.41 &  1.10e-4 &  2.08\\
   &32&  7.61e-5 &  2.04 &  1.03e-4 &  2.04 &  3.30e-3 &  1.16 &  2.69e-5 &  2.03\\
\hline
 \end{tabular}
 \end{threeparttable}
\end{table}

 \begin{table}[htbp]\caption{Various errors and corresponding convergence rates  for $k=1,2$ with $l=k-1$ and $\epsilon=10^{-10}$ on  $\Omega_1$.}\label{7}\centering
\begin{threeparttable}
        \begin{tabular}{c |c |c | c | c | c | c | c | c | c  |}
        \hline
 & $1/h$ &   $\|e_b\|$ & rate& $\|e_n\|$  & rate & $\|\nabla e_0\|$ & rate & $\|e_0\|$ & rate \\
       \hline \cline{2-10}
      &4&  2.96e-01&      -- &  1.54e-0 &    --&  1.09e-0 &      --&  1.09e-1 &  -- \\
      &8&  1.80e-01&   0.73 &  7.93e-1 &  0.96 &  5.62e-1 &  0.95 &  6.43e-2 &  0.75\\
$k=1$ &16&  1.06e-01&   0.76 &  4.06e-1&   0.97 &  3.25e-1 &  0.79 &  3.77e-2 &  0.77\\
      &32&  5.86e-02&   0.85 &  2.06e-1 &  0.98 &  2.05e-1 &  0.67 &  2.08e-2 &  0.86\\
      &64&  3.10e-02 &  0.92 &  1.04e-1 &  0.99 &  1.35e-1 &  0.66&  1.10e-2 &  0.92\\
\hline
& $1/h$ &   $\|e_b\|$ & rate& $\|e_n\|$  & rate & $\|\nabla e_0\|$ & rate & $\|e_0\|$ & rate \\
       \hline \cline{2-10}
         &2&   9.82e-2 &   -- &   2.77e-1  &     -- &   4.25e-1&    -- &  3.98e-2  &   --\\
        &4&   2.81e-2  &  1.80 &  7.47e-2  &  1.89  &   2.07e-1&   1.04 &  1.04e-2  &  1.94 \\
 $k=2$  &8&  7.93e-3   & 1.82 &  1.93e-2  &  1.95   &   1.14e-1&   0.86  &  2.84e-3  &  1.86 \\
        &16&  2.11e-3 &   1.91 &   4.92e-3  &  1.98 &   6.13e-2&   0.89 &   7.59e-4 &   1.91\\
        &32&   5.44e-4 &   1.96 &   1.24e-3  &  1.99 &   3.20e-2&  0.94 &   1.97e-4 &   1.95\\

\hline
 \end{tabular}
 \end{threeparttable}
\end{table}

\begin{table}[htbp]\caption{Various errors and corresponding convergence rates  for $k=2$ with  $\epsilon=10^{-2}$ on  $\Omega_1$.}\label{8-1}\centering
\begin{threeparttable}
        \begin{tabular}{c |c |c | c | c | c | c | c | c | c  |}
        \hline
 & $1/h$ &   $\|e_b\|$ & rate& $\|e_n\|$  & rate & $\|\nabla e_0\|$ & rate & $\|e_0\|$ & rate \\
       \hline \cline{2-10}
      &2&  4.31e-02 &    -- &  6.79e-2  &    -- &  2.44e-1  &   -- &  1.59e-2 &     --\\
      &4&  6.18e-03 &  2.80 &  8.97e-3 &  2.92 &  6.47e-2  & 1.92 &  2.02e-3 &  2.98\\
$l=k$ &8&  7.98e-04 &  2.95 &  1.25e-3 &  2.84 &  1.73e-2 &  1.90 &  2.55e-4  & 2.99\\
      &16&  9.42e-05 &  3.08 &  1.94e-4  & 2.69 &  4.31e-3 &  2.00 &  3.12e-5  & 3.03\\
      &32&  1.11e-05 &  3.08 &  3.56e-5  & 2.44 &  1.03e-3 &  2.06 &  3.72e-6  & 3.07\\
\hline
 & $1/h$  &   $\|e_b\|$ & rate& $\|e_n\|$  & rate & $\|\nabla e_0\|$ & rate & $\|e_0\|$ & rate \\
       \hline \cline{2-10}
       &4& 2.71e-2 &    --&  7.51e-2  &    -- &  1.96e-1 &     -- &  1.02e-2 &   --   \\
        &8& 7.01e-3 &  1.95 &  1.94e-2 &  1.96 &  9.48e-2 &  1.05 &  2.64e-3 &  1.95 \\
 $l=k-1$ &16& 1.44e-3 &  2.29 &  4.77e-3  & 2.02 &  3.63e-2 &  1.38 &  5.59e-4 &  2.24 \\
         &32& 2.05e-4 &  2.81 &  1.13e-3  & 2.07 &  9.36e-3 &  1.96  & 8.21e-5 &  2.77 \\
        &64& 1.97e-5 &  3.38 &  2.73e-4  & 2.05 &  1.64e-3 &  2.51 &  8.01e-6 &  3.36 \\
 \hline
 \end{tabular}
 \end{threeparttable}
\end{table}

\begin{table}[htbp]\caption{Various errors and corresponding convergence rates  for $k=1,2$ with $l=k$ and $\epsilon=10^{-2}$ on  $\Omega_2$.}\label{6-1}\centering
\begin{threeparttable}
        \begin{tabular}{c |c |c | c | c | c | c | c | c | c  |}
        \hline
 & $1/h$ &   $\|e_b\|$ & rate& $\|e_n\|$  & rate & $\|\nabla e_0\|$ & rate & $\|e_0\|$ & rate \\
       \hline \cline{2-10}
 & 4 &1.00e-1 &      --  & 1.32e-1  &   -- &  7.45e-1  &  --  & 3.23e-2 &     -- \\
 & 8 &3.35e-2 &  1.58  & 4.57e-2  & 1.53 &  3.97e-1  &0.90 & 1.12e-2 & 1.53 \\
$k=1$& 16& 1.02e-2 &  1.72  & 1.47e-2 &  1.63 &  1.97e-1  & 1.01&  3.56e-3 & 1.66 \\
 & 32&  2.58e-3&   1.99&  4.12e-3 &  1.84 &  9.60e-2 & 1.03 & 9.18e-4 & 1.96 \\
& 64 & 6.43e-4 &  2.00 &  1.32e-3 &  1.64 &  4.76e-2 &  1.01   &2.30e-4 & 2.00 \\
 \hline
  & $1/h$ &   $\|e_b\|$ & rate& $\|e_n\|$  & rate & $\|\nabla e_0\|$ & rate & $\|e_0\|$ & rate \\
       \hline \cline{2-10}
& 4  & 6.03e-3 &        --&  8.94e-3 &     -- &  6.18e-2 &    --&   1.97e-3 &   --        \\
& 8  &8.85e-4  & 2.77  & 1.39e-3 &  2.68  & 1.72e-2  &1.85 & 2.85e-4  & 2.79 \\
$k=2$& 16 & 1.02e-4 &  3.12  & 2.02e-4 &  2.78 &  4.27e-3 & 2.01 &  3.23e-5  & 3.14\\
 & 32 & 1.15e-5 &  3.15 &  3.58e-5 &  2.50 &  1.02e-3 & 2.07 &  3.64e-6  & 3.15\\
& 64 & 1.38e-6 &  3.06 &  8.00e-6 &  2.16 &  2.50e-4 & 2.03 &  4.36e-7  &3.06\\
\hline
 \end{tabular}
 \end{threeparttable}
\end{table}

 \begin{table}[htbp]\caption{Various errors and corresponding convergence rates  for $k=1,2$ with $l=k-1$ and $\epsilon=10^{-2}$ on   $\Omega_2$.}\label{7-1}\centering
\begin{threeparttable}
        \begin{tabular}{c |c |c | c | c | c | c | c | c | c  |}
        \hline
 & $1/h$ &   $\|e_b\|$ & rate& $\|e_n\|$  & rate & $\|\nabla e_0\|$ & rate & $\|e_0\|$ & rate \\
       \hline \cline{2-10}
  & 4 & 1.89e-1 &     -- &   1.44e-0 &    --  & 9.87e-1&     --  & 7.02e-2 &     --\\
& 8 & 1.06e-1 &  0.83  & 7.33e-1&   0.98  & 4.91e-1 & 1.00 &  3.83e-2 &0.87  \\
$k=1$& 16 & 6.53e-2 & 0.72  & 3.72e-1&   0.98  & 2.71e-1 &0.86 &  2.33e-2 & 0.72  \\
& 32 & 3.52e-2 &  0.89  & 1.87e-1&   0.99  & 1.54e-1  &0.81 &  1.25e-2 & 0.90 \\
&64  & 1.54e-2 &  1.20  & 9.28e-2&   1.01  & 7.85e-2  &0.98 &  5.45e-3 & 1.20 \\
 \hline
 & $1/h $&   $\|e_b\|$ & rate& $\|e_n\|$  & rate & $\|\nabla e_0\|$ & rate & $\|e_0\|$ & rate \\
       \hline \cline{2-10}
&4 & 2.49e-2   &    --  &  7.00e-2    &   --   & 1.76e-1    &    --   & 9.52e-3   &     -- \\
& 8  & 6.56e-3  &  1.92  &  1.81e-2   & 1.95  &  8.61e-2   & 1.03 &  2.49e-3   & 1.93 \\
$k=2$& 16  & 1.34e-3  &  2.29 & 4.45e-3  &  2.02  &  3.32e-2  & 1.38 &  5.24e-4   & 2.25 \\
& 32  & 1.89e-4  &  2.83  &  1.05e-3  &  2.08   & 8.52e-3  & 1.96 & 7.54e-5  & 2.80\\
& 64  & 1.80e-5  &  3.39 &  2.54e-4  &  2.05   & 1.49e-3  & 2.51 &  7.25e-6  & 3.38 \\
\hline
 \end{tabular}
 \end{threeparttable}
\end{table}

In what follows of this section, we present the plot of the numerical solution $u_h$ arising from the primal-dual weak Galerkin scheme \eqref{32}-\eqref{2} for test problems for which the exact solution is not known. 

 {\bf Test Example 4.} The diffusion $a$ is given by $a_{11}=a_{22}=10^{-4}$, $a_{12}=a_{21}=0$;  the convection is set as ${\bf b}=(y, -x)^T$, the domain is an unit square domain $\Omega_1$; and the mixed boundary conditions are $g_1=\sin(3x)$ on the inflow boundary $\Gamma_D=\{(x,y): {\bf b}\cdot {\bf n} <0\}$ and $g_2=0$ on $\Gamma_N=\partial\Omega_1\setminus \Gamma_D$. Figure \ref{fig:2-1} presents the plots for the numerical solution $u_h$ obtained from the PDWG numerical method \eqref{32}-\eqref{2} with $k=1$ and $l=k-1$
for the convection-dominated  diffusion problem when different load functions  $f=1$ (left) and $f=0$ (right) are employed, respectively.

\begin{figure}[htb]
  \centering
  \includegraphics[width=.32\textwidth]{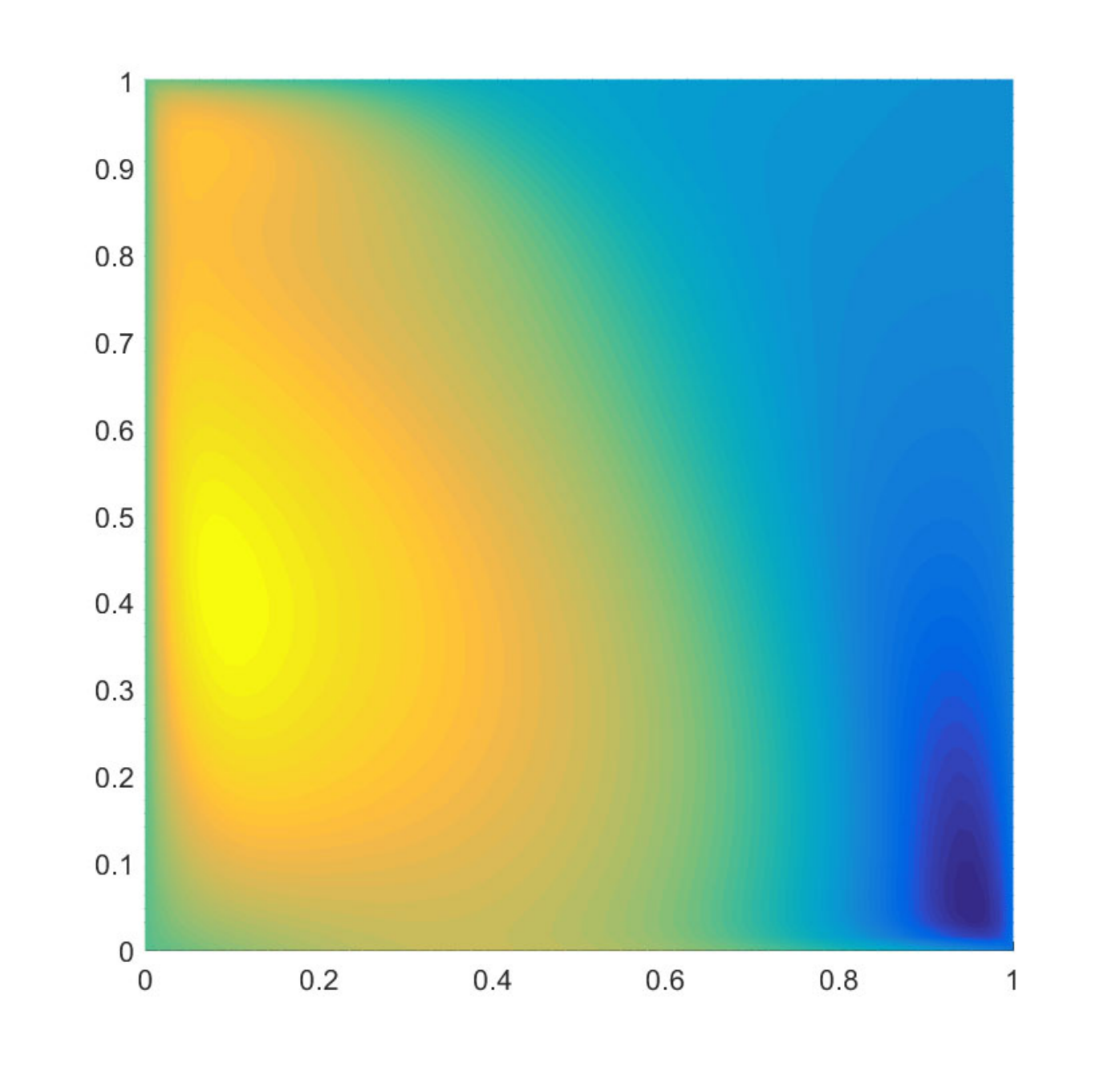}
  \includegraphics[width=.32\textwidth]{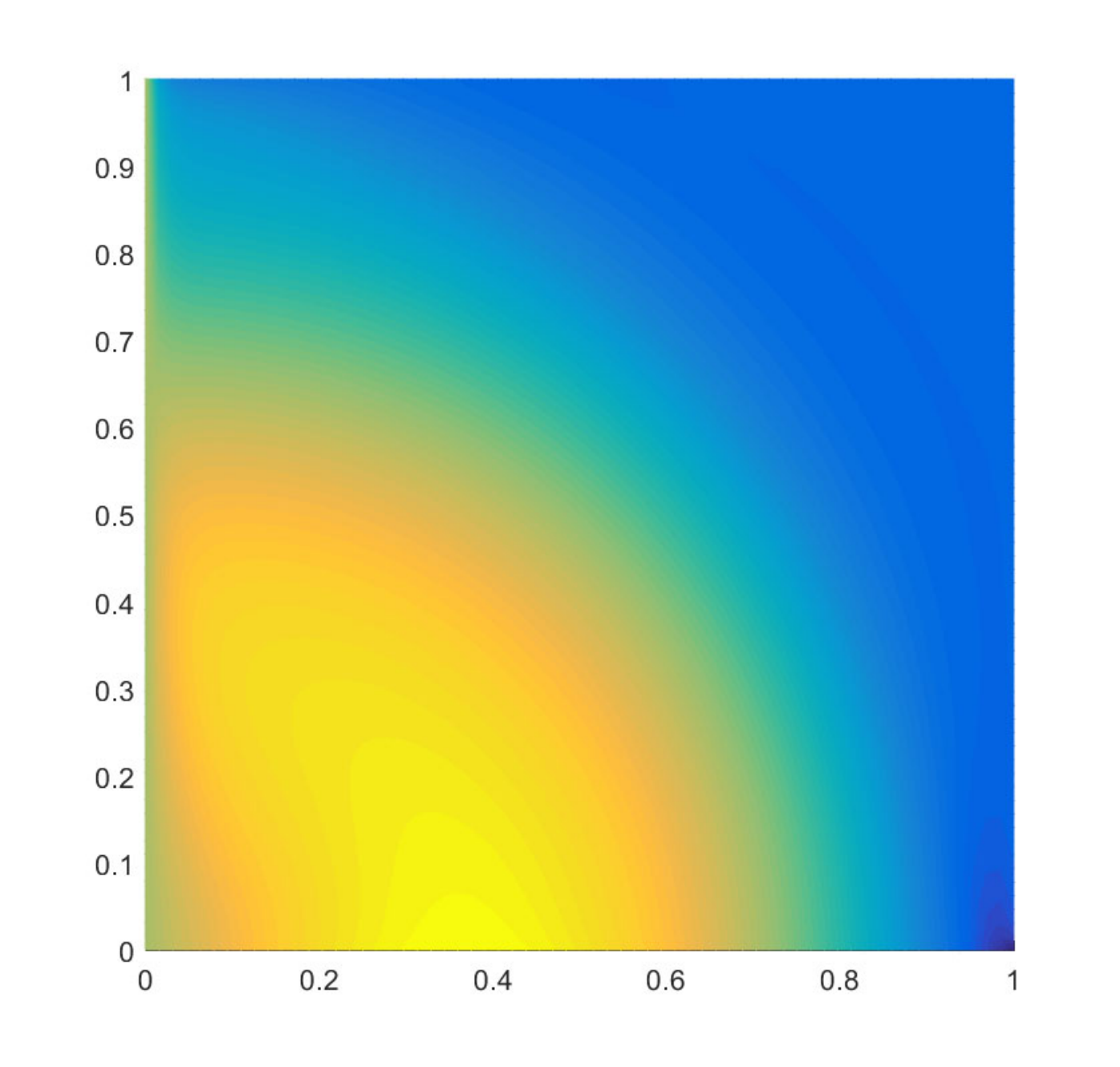}~
 \caption{Contour plots of the numerical solution $u_h$ on $\Omega_1$ for  the load functions $f=1$ (left) and $f=0$ (right).}
  \label{fig:2-1}
\end{figure}

 {\bf Test Example 5.} The diffusion is given by
 $a_{11}=a_{22}=10^{-5}, a_{12}=a_{21}=0$; and
 the convection vector is set as ${\bf b}=(y,-x)^T$. For the unit square domain $\Omega_1$, $\Gamma_N=\{(x,y): x=1\ {\rm or } \ y=0\}$; and for the L-shaped domain $\Omega_2$, $\Gamma_N=\{(x, y): x=1\  {\rm or  } \ y=0.5\}$. The mixed boundary conditions are $g_1=\sin(2x)$ on $\Gamma_D=\partial\Omega\setminus\Gamma_N$ and $g_2=0$ on $\Gamma_N$. We take $l=k-1$ and $k=1$. Figures \ref{fig:33-1}-\ref{fig:34-1} demonstrate the numerical solutions $u_h$ on the unit square domain $\Omega_1$ and the L-shaped domain $\Omega_2$ when different
load functions $f=0$ (left) and $f=1$ (right) are employed, respectively.

\begin{figure}[htb]
  \centering
  \includegraphics[width=.32\textwidth]{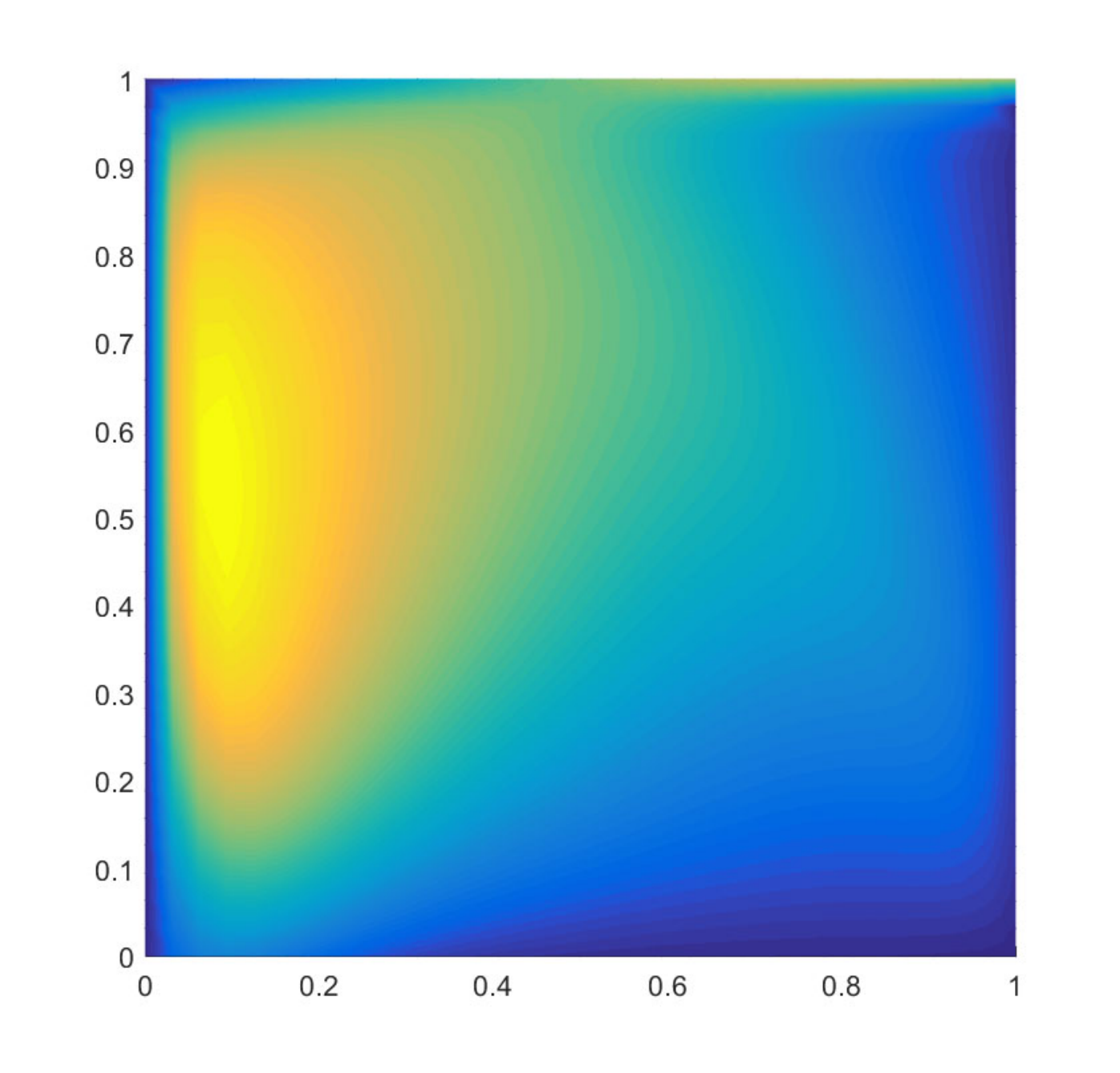}
  \includegraphics[width=.32\textwidth]{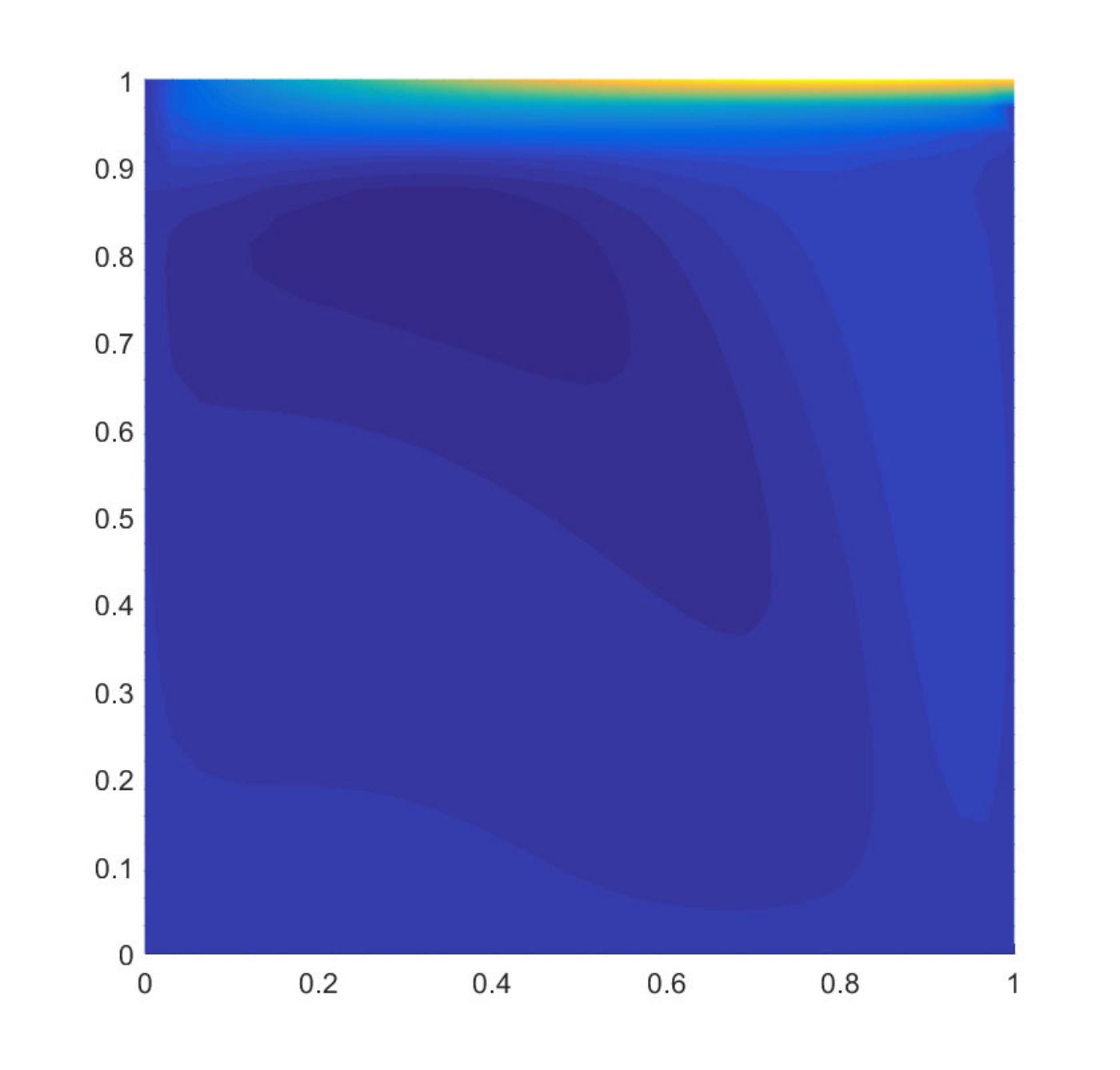}~
 \caption{Contour plots of the numerical solution $u_h$ on   $\Omega_1$ with
  the load functions $f=1$ (left) and $f=0$ (right).}
  \label{fig:33-1}
\end{figure}

\begin{figure}[htb]
  \centering
  \includegraphics[width=.32\textwidth]{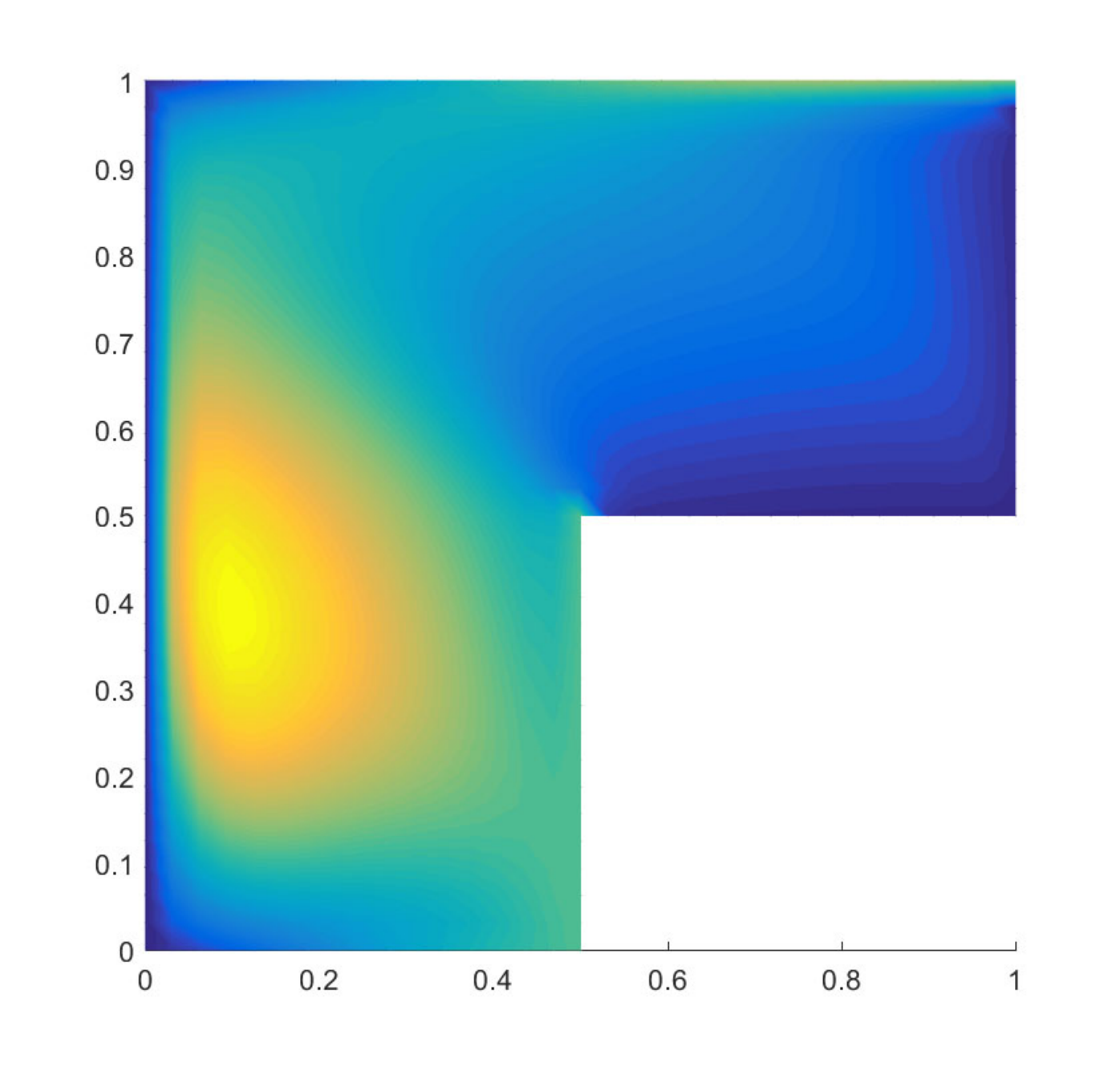}
  \includegraphics[width=.32\textwidth]{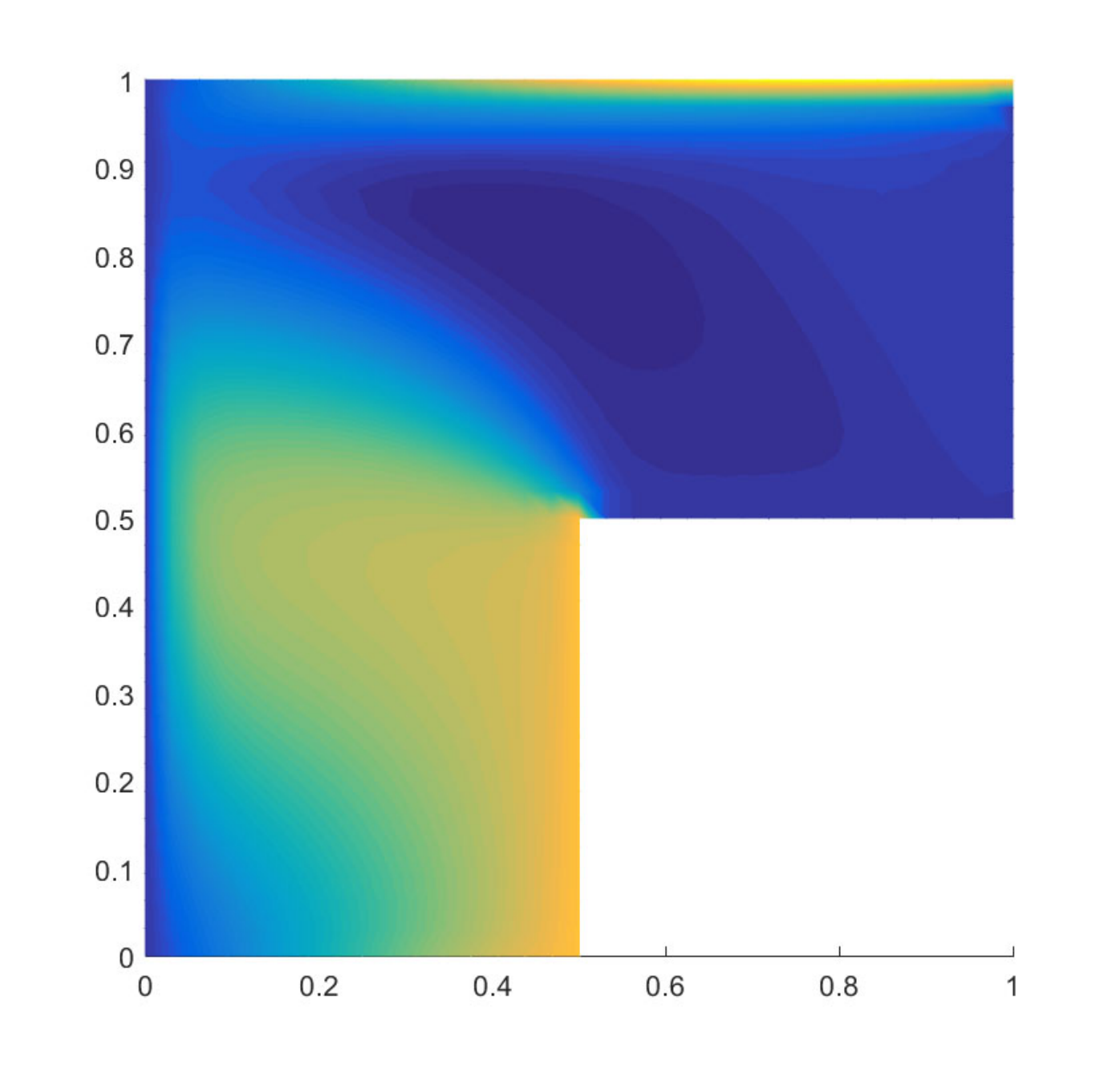}~
 \caption{Contour plots of the numerical solution $u_h$ on   $\Omega_2$ with the load functions $f=1$ (left) and $f=0$ (right).}
\label{fig:34-1}
\end{figure}

\vfill\eject

\end{document}